\documentclass[12pt]{scrartcl}
\usepackage[english]{babel}
\usepackage[utf8x]{inputenc}
\usepackage{amsfonts}
\usepackage{amsmath}
\usepackage{amssymb}
\usepackage{amsthm}
\usepackage{mathrsfs}
\usepackage[colorlinks=false,urlcolor=blue]{hyperref}
\usepackage{breakurl}
\usepackage{listings}
\usepackage{graphicx}
\usepackage{enumerate}
\usepackage{mathtools}

\addtokomafont{section}{\rmfamily}
\addtokomafont{subsection}{\rmfamily}
\addtokomafont{paragraph}{\rmfamily}
\usepackage{indentfirst}

\swapnumbers
\theoremstyle{plain}
\newtheorem{satz}{Theorem}[section]
\newtheorem{lem}[satz]{Lemma}
\newtheorem{kor}[satz]{Corollary}
\newtheorem{prop}[satz]{Proposition}
\theoremstyle{definition}
\newtheorem{defn}[satz]{Definition}
\newtheorem{bem}[satz]{Remark}

\newcommand{\R}{\mathbb{R}}
\newcommand{\C}{\mathbb{C}}
\newcommand{\N}{\mathbb{N}}
\newcommand{\CP}{\mathbb{CP}}
\renewcommand{\O}{\mathrm{O}}
\newcommand{\Id}{\mathrm{Id}}
\newcommand{\Ric}{\mathrm{Ric}}
\newcommand{\scal}{\mathrm{scal}}
\newcommand{\vol}{\mathrm{vol}}
\newcommand{\Sym}{\operatorname{Sym}}
\newcommand{\tr}{\operatorname{tr}}
\newcommand{\im}{\operatorname{im}}
\newcommand{\diag}{\operatorname{diag}}
\newcommand{\End}{\operatorname{End}}
\newcommand{\Hom}{\operatorname{Hom}}
\newcommand{\Ad}{\operatorname{Ad}}
\newcommand{\ad}{\operatorname{ad}}
\newcommand{\Cas}{\mathrm{Cas}}
\newcommand{\Sy}{\mathscr{S}}
\newcommand{\X}{\mathfrak{X}}
\newcommand{\U}{\operatorname{U}}
\newcommand{\SU}{\operatorname{SU}}
\newcommand{\su}{\mathfrak{su}}
\renewcommand{\u}{\mathfrak{u}}
\newcommand{\SO}{\operatorname{SO}}
\newcommand{\so}{\mathfrak{so}}
\newcommand{\Sp}{\operatorname{Sp}}
\newcommand{\GL}{\operatorname{GL}}
\newcommand{\gl}{\mathfrak{gl}}
\makeatletter
\@ifundefined{sl}{%
\newcommand{\sl}{\mathfrak{sl}}
}{
\renewcommand{\sl}{\mathfrak{sl}}
}
\makeatother
\newcommand{\spann}{\operatorname{span}}
\newcommand{\pr}{\operatorname{pr}}
\renewcommand{\i}{\mathrm{i}}
\newcommand{\m}{\mathfrak{m}}
\newcommand{\h}{\mathfrak{h}}
\newcommand{\g}{\mathfrak{g}}
\renewcommand{\k}{\mathfrak{k}}
\renewcommand{\t}{\mathfrak{t}}
\DeclareMathOperator*{\closedsum}{\overline{\bigoplus}}
\newcommand{\isom}{\mathfrak{iso}}
\newcommand{\bbE}{\mathbb{E}}
\newcommand{\curvop}{\mathcal{R}}
\newcommand{\bbF}{\mathbb{F}}
\newcommand{\Ltwoinprod}[2]{\mathchoice
  {\big(#1,#2\big)_{L^2}}%
  {(#1,#2)_{L^2}}%
  {(#1,#2)_{L^2}}%
  {(#1,#2)_{L^2}}%
}
\newcommand{\prim}{0}
\newcommand{\curvendo}{\mathcal{K}(R)}
\newcommand{\LL}{\Delta_{\mathrm{L}}}
\newcommand{\quadric}{\mathscr{Q}}
\newcommand{\intprod}{\mathbin{\lrcorner}}
\newcommand{\Eop}{\mathrm{E}}
\newcommand{\rmE}{\mathrm{E}}
\newcommand{\rmF}{\mathrm{F}}
\newcommand{\rmS}{\mathrm{S}}

\title{\rmfamily On the rigidity of the complex Grassmannians}
\author{Paul Schwahn\footnote{Laboratoire de Math\'ematiques d'Orsay, Universit\'e Paris-Saclay, Rue Michel Magat 307, 91405 Orsay Cedex, France}, Uwe Semmelmann\footnote{Institut f\"ur Geometrie und Topologie, Fachbereich Mathematik, Universit\"at Stuttgart, Pfaffenwaldring 57, 70569 Stuttgart, Germany.}}
\date{\today}

\begin{document}

\maketitle

\begin{abstract}
\noindent
We study the integrability to second order of the infinitesimal Einstein de\-for\-ma\-tions of the symmetric metric $g$ on the complex Grassmannian of $k$-planes inside $\C^n$. By showing the nonvanishing of Koiso's obstruction polynomial, we characterize the infinitesimal deformations that are integrable to second order as an explicit variety inside $\su(n)$. In particular we show that $g$ is isolated in the moduli space of Einstein metrics if $n$ is odd.

\medskip


\medskip

\end{abstract}

\section{Introduction}
\label{sec:intro}

Let $g$ be an Einstein metric on a closed manifold $M^m$, that is $\Ric_g = E g$, where $E= \frac{\scal_g}{m}$ is the Einstein constant. An important problem is to decide whether $g$ is \emph{rigid}, i.e.~whether the equivalence class $[g]$ under homothetic scaling and pullback by diffeomorphisms is isolated in the moduli space of Einstein structures. By a deep result of Koiso \cite[Thm.~3.2]{Koiso83}, $g$ is not rigid if and only if there exists an analytic curve $(g_t)$ of Einstein metrics through $g=g_0$ that are not all equivalent. In fact, this curve may be taken to consist of unit volume metrics and to be transversal to  orbits of the diffeomorphism group. Thus one may assume the tangent vector $h=\dot g_0$ to be a so-called \emph{tt}-tensor, i.e.~$h$ is \textbf{t}raceless ($\tr h=0$) and \textbf{t}ransverse (divergence-free, $\delta h=0$). The symmetric $2$-tensor $h=\dot g_0$ then satisfies the linearized Einstein equation $\LL h = 2E h$, where $\LL$ is the Lichnerowicz Laplacian. The solutions to this equation which are in addition tt-tensors are called \emph{infinitesimal Einstein deformations}. The finite-dimensional space of infinitesimal Einstein deformations of $g$ is denoted by $\varepsilon (g)$. Canonically, the isometric group of $(M,g)$ acts on $\varepsilon(g)$. For a comprehensive exposition, we refer the reader to the monograph of Besse \cite[\S12]{besse}.

It follows immediately from Koiso's results that $g$ is rigid if $\varepsilon(g)$ is trivial. However the converse is in general false, as illustrated by the symmetric metric on $\CP^{2k}\times\CP^1$ \cite{Koiso82}. The question is thus: given  $h \in \varepsilon(g)$, does there exists a curve of Einstein metrics $(g_t)$ with $h=\dot g_0$? If this is the case, the infinitesimal deformation $h$ is called \emph{integrable}.

The main tool for studying this question is the Einstein operator
\[\Eop(g) : = \Ric_g - \frac{\int_M \scal_g \vol_g}{m} g,\]
where the volume  of $g$ is normalized to one. If $(g_t)$ is an analytic curve of Einstein metrics of volume one, then $\Eop(g_t)$ and all its derivatives have to vanish at $t=0$. This gives a sequence of obstructions for an arbitrary curve through $g$ to be a curve of Einstein metrics. An element $h \in \varepsilon (g)$ is called \emph{integrable
to order $k$} there exists symmetric $2$-tensors $h_2,\ldots,h_k$ such that for the curve $(g_t)$ defined by
\[g_t=g+th+\sum_{j=2}^k\frac{t^j}{j!}h_j,\]
the first $k$ derivatives of $\Eop(g_t)$ vanish at $t=0$, i.e.
\[
 \frac{d^j}{d  t^j}\Big|_{t=0} \Eop(g_t)  = 0
 \qquad
 \text{for } j=1, \ldots , k.
\]
An element $h \in \varepsilon (g)$ is called \emph{formally integrable} if it is integrable to all orders, that is, if there is a formal power series solution of $E(g_t)=0$. Koiso showed that this is in fact equivalent to integrability, i.e.~the existence of an actual curve of Einstein metrics through $g$ with first order jet $h$. In particular, if $h$ is not integrable to order $k$ for some $k\in\N$, then there does not exist a nontrivial curve of Einstein metrics through $g$ and thus the metric $g$ is rigid.

By definition, any $h \in \varepsilon(g)$ is integrable to first order, i.e.~the curve $g_t = g + th$ solves the Einstein equation to first order in the sense that the first derivative of $\Eop(g_t)$ vanishes at $t=0$. The next step is to study the question whether a given $h\in \varepsilon (g)$ is
integrable to second order. For this, one has to find a symmetric tensor $h_2$ such that the curve $g_t= g + t h + \frac{t^2}{2} h_2$ satisfies
\begin{equation}
\label{2ndobstr}
 0 = \frac{d^2}{d  t^2}\Big|_{t=0} \Eop(g_t) = \Eop''_g(h,h) + \Eop'_g(h_2),
\end{equation}
that is, $\Eop''_g(h,h) \in \im\Eop'_g$. Koiso showed that this is equivalent to $\Eop''_g(h,h) $
being $L^2$-orthogonal to the space $\varepsilon (g)$ itself \cite[Prop.~3.2]{Koiso82}. Hence $h \in \varepsilon (g)$ is integrable to second order if and only if
$\Psi(h,h,k) = 0$ for all $k \in \varepsilon (g)$, where the trilinear form $\Psi $ is defined  by
\[
\Psi(h_1, h_2, h_3) :=  \Ltwoinprod{\Eop''_g(h_1, h_2)}{ h_3}\qquad
 \text{for } h_1, h_2, h_3 \in \varepsilon (g).
\]
Thus a necessary condition for integrability to second order for an $h\in\varepsilon(g)$ is the vanishing of the \emph{Koiso obstruction} $\Psi(h) = \Psi(h,h,h)$. Koiso gives an explicit formula for $\Psi(h)$, which is a complicated combination of second order derivatives of $h$ \cite[Lem.~4.3]{Koiso82}, which is in general rather difficult to compute.

In \cite[Thm.~1.1]{NS23} an explicit formula for the full obstruction form $\Psi(h_1, h_2, h_3)$ is given in terms of the Frölicher-Nijenhuis bracket in a coordinate-free way. For the special case $h_1=h_2=h_3=h$ it recovers the Koiso obstruction $\Psi(h)$. In particular, it turns out that $\Psi$ is symmetric in all 3 arguments. However, it can be shown a priori that on $\varepsilon(g)$, the form $\Psi$ coincides with the third variation of the Einstein--Hilbert functional at $g$ and is thus completely symmetric. Hence the new formula for $\Psi(h_1,h_2,h_3)$ can also be checked by reformulating the Koiso obstruction and then polarizing.

Let $(M, g, J)$ be a compact Kähler--Einstein manifold with Kähler form $\omega$. In this case, a further simplification of the expression for $\Psi$ is possible. Identifying symmetric $2$-tensors with symmetric endomorphisms of $TM$ via the metric, one has a splitting of $\Sym^2T^\ast M$ into endomorphisms commuting resp.~anti-commuting with $J$, and a corresponding splitting
\[\varepsilon(g) = \varepsilon^+(g) \oplus \varepsilon^-(g).\]
Koiso showed in \cite{Koiso83} that
\[
\varepsilon^+(g) = \{ F\circ J \,|\, F \in \Omega^{1,1}_0(M) \cap \ker(\Delta - 2E) \cap \ker d^\ast \} ,
\]
 where $\Omega^{1,1}_0(M) $ is the space of primitive $(1,1)$-forms and $\Delta = d d^* + d^*d$ is the Hodge Laplacian.
 Now \cite[Thm.~1.4]{NS23} gives the following criterion:\footnote{The original version contained a numerical error: the factor $8E$ originally read $\frac{17E}{2}$.}
\begin{prop}
\label{intkaehler}
 Assume that $(M,g,J)$ is a compact Kähler--Einstein manifold with $\varepsilon^-(g)=0$. Then
 $h = F \circ J \in \varepsilon(g)$ is integrable to second order if and only if
\[
 \Ltwoinprod{\omega \wedge dG}{ F \wedge dF } + \Ltwoinprod{\omega \wedge dF}{F \wedge dG + G \wedge dF } = 8E\Ltwoinprod{F^2 \circ J}{ G}
\]
holds for all elements $G\circ J \in  \varepsilon^+(g) $. In particular, the Koiso obstruction for $h = F \circ J $ can be written as
\begin{equation}
\label{psikaehler}
\Psi(h) = 6 \Ltwoinprod{\omega \wedge dF}{F \wedge dF } -8E\Ltwoinprod{F \wedge F}{ F \wedge \omega}.
\end{equation}
\end{prop}

In \cite{Koiso80} Koiso already studied the case of compact symmetric spaces $M=G/K$ and arrived at the following list of infinitesimally deformable spaces:

\begin{prop}
\label{koisoied}
Only the following irreducible symmetric spaces of compact type admit infinitesimal Einstein deformations:
\begin{enumerate}[\upshape(i)]
\item
$M= \SU(n)$ \quad with $ \; n\ge 3$,
\item
$M= \SU(n)/\SO(n)$ \quad with $  \; n\ge 3$,
\item
$M= \SU(2n)/\Sp(n)$ \quad with $  \; n\ge 3$,
\item
$M=\SU(p+q)/\rmS(\U(p) \times \U(q))$  \quad with $  \; p,q \ge 2$,
\item
$M=\rmE_6/\rmF_4$.
\end{enumerate}
\end{prop}

In all of these cases it turned out that $\varepsilon(g) \cong \g$ as a $G$-module, where $\g \cong \isom(M,g)$ is the Lie algebra of $G$, which is also the isometry group of $(M,g)$. Moreover, in these cases the Koiso obstruction can be viewed as an invariant cubic
polynomial on $\g$, i.e. as an element of $(\Sym^3 \g^\ast)^G$.

Koiso already showed that all infinitesimal Einstein deformations on $M=E_6/F_4$ are unobstructed to second order simply because in this case $(\Sym^3 \g^\ast)^G=0$. The first substantial progress on the rigidity problem for compact symmetric spaces was recently obtained in \cite{BHMW}. The authors show that the  bi-invariant metric on $\SU(n)$ is rigid for $n$ odd, thereby exhibiting the first example of a rigid Einstein metric on an irreducible manifold. In addition, for $n$ even they describe the variety of infinitesimal Einstein deformations integrable to second order.

In this article we will study the complex Grassmannians $M=\SU(p+q)/\rmS(\U(p) \times \U(q))$, with $ p,q \ge 2$. These are the only Hermitian symmetric spaces in the list above. Moreover, all other spaces in the list admit an invariant symmetric $3$-tensor \cite[Prop.~2.1]{GGiso}, which can be used to give an explicit parametrization of the space $\varepsilon(g)$ by Killing vector fields through a bundle morphism \cite[Lem.~2.8]{BHMW}. For the Grassmannians we need another way to make the isomorphism $\g\cong\varepsilon(g)$ explicit. Extending and generalizing the approach of \cite{NS23}, we construct it as a first-order differential operator instead of a bundle map. The construction in \cite{NS23} relies on the quaternionic Kähler structure of the two-plane Grassmannians, however our approach shows that this is not a crucial ingredient.

We note that the complex Grassmannians satisfy the assumptions of Proposition~\ref{intkaehler} since the space $\varepsilon^-(g)$ vanishes, as will become evident later. This enables us to make use of the simpler form \eqref{psikaehler} for the obstruction polynomial.

In \cite[Thm.~1.5]{NS23} it is shown that the symmetric metric on the $2$-plane Grassmannian $\SU(n+2)/\rmS(\U(2)\times\U(n))$ is rigid for $n$ odd. More generally the set of infinitesimal Einstein deformations integrable to second order is explicitly
characterized as an algebraic subset of $\g = \su(n+2)$. We prove a similar statement for all Grassmannians, simplifying the argument, and recovering the results of \cite{NS23} in the special case of the $2$-plane Grassmannian. Our main result is the following.

\begin{satz}
\label{main}
Let $M = \SU(n)/\rmS(\U(n_+) \times \U(n_-))$, $n=n_+ + n_-$ with $n_+,n_-\ge 2$, the complex Grassmannian with its Hermitian symmetric
structure $(M,g,J)$. Then the following holds:
\begin{enumerate}[\upshape(i)]
\item
The set of infinitesimal Einstein deformations in $\varepsilon(g)$ which are integrable to second order is isomorphic to the
variety
\[
\quadric = \left\{X \in \mathfrak{su}(n) \, \middle| \, X^2 = \frac{\tr(X^2)}{n} I_n  \right\}.
\]
\item
If $n$ is odd, then all infinitesimal Einstein deformations are obstructed to second order. Thus the metric $g$ is rigid.
\end{enumerate}
\end{satz}

The structure of the article is as follows. \S\ref{sec:grassmann} recalls some basic facts about the structure of the complex Grassmannians as a homogeneous space and introduces some notation. In \S\ref{sec:fibrewise} we shall discuss the required fibrewise machinery, that is, considerations of representation theory and curvature, and some auxiliary invariant objects are introduced. After a brief discussion of Killing fields and potentials in \S\ref{sec:diff} together with some differential identities, we explicity parametrize the infinitesimal Einstein deformations by Killing vector fields in \S\ref{sec:ied}.

Before we can take on the Koiso obstruction polynomial $\Psi$, we collect some of the properties of and identities between invariant cubic forms on $\su(n)$ resp.~on Killing fields in \S\ref{sec:invforms}. Finally, in \S\ref{sec:obstruction}, we express $\Psi$ in terms of previously defined polynomials and prove Theorem~\ref{main}.

Returning to the list in Proposition~\ref{koisoied}, we note that the only remaining spaces where one can expect to show rigidity are $\SU(n)/\SO(n)$ with $n\geq3$ odd. Indeed, for $G=\SU(n)$ with $n$ even, Propositions~\ref{invcubic} and \ref{variety} show that there always exist infinitesimal Einstein deformations that are integrable to second order, regardless of whether the obstruction polynomial $\Psi$ vanishes identically or not. Therefore, in order to investigate the rigidity of the remaining symmetric spaces, one would have to take into account at least the \emph{third order obstruction}. That is, given an infinitesimal deformation $h\in\varepsilon(g)$ together with a symmetric tensor $h_2$ satisfying the second order condition \ref{2ndobstr}, by differentiating $\Eop(g_t)$ thrice and applying Koiso's orthogonality argument \cite[Prop.~3.2]{Koiso82}, one needs to check whether
\[\Eop'''_g(h,h,h)+3\Eop''_g(h,h_2)\perp\varepsilon(g).\]
A closed formula for $\Eop'''_g$, which would involve the third variation of the Ricci tensor, is however currently lacking.

\section{The structure of the complex Grassmannians}
\label{sec:grassmann}

Let $n_+,n_-\geq 2$ and $n=n_++n_-$. We consider in this article the complex Grassmannians
\[M=G/K=\frac{\SU(n)}{\rmS(\U(n_+)\times\U(n_-))},\]
a symmetric space of real dimension $2n_+n_-$. Denoting with $\g$ and $\k$ the Lie algebras of $G$ and $K$, respectively, we have a reductive decomposition $\g=\k\oplus\m$ that satisfies the Cartan relation $[\m,\m]\subset\k$. The isotropy representation $\m$, canonically identified with the tangent space $T_oM$ at the identity coset, is as a complex representation of $K$ isomorphic to $\C^{n_+}\otimes_\C(\C^{n_-})^\ast$, where $\C^k$ denotes the definining representation of $\U(k)$.

Equipped with the standard metric $g$ induced by the negative of the Killing form (see~\S\ref{sec:casimir}) and the canonical complex structure $J$ corresponding to multiplication with $\i$ on $\m$, the manifold $(M,g,J)$ is an irreducible Hermitian symmetric space, and in particular Kähler--Einstein of positive scalar curvature.

Since $(M,g)$ is a symmetric space, the Lie algebra of Killing vector fields $\isom(M,g)$ is isomorphic to $\g$. The isomorphism is explicitly given by taking \emph{fundamental vector fields}, that is
\[\g\ni X\longmapsto\tilde X,\qquad \tilde X_p:=\frac{d}{dt}\Big|_{t=0}\exp(tX)p.\]

We frequently and tacitly identify $2$-forms $\alpha\in\Lambda^2T^\ast_pM$ via the metric with skew-symmetric endomorphisms $\alpha\in\so(T_pM)$, and we do not notiationally distinguish between those two. Concretely, the identification is given by
\[g(\alpha X,Y):=\alpha(X,Y),\qquad X,Y\in T_pM.\]
In particular, $J$ is identified with the Kähler form $\omega$.

Let $\Lambda^{1,1}T^\ast M$ be the (real) subbundle of $(1,1)$-forms, i.e.~$2$-forms commuting with $J$, and let
\[\Lambda^{1,1}_0T^\ast M=\{\alpha\in\Lambda^{1,1}T^\ast M\,|\,\tr(\alpha J)=0\}\]
be the subbundle of \emph{primitive} $(1,1)$-forms. We denote with $\alpha_0$ the projection to $\Lambda^{1,1}_0T^\ast M$ of any $\alpha\in\Lambda^{1,1}T^\ast M$. Composition with the endomorphism $J$ yields a parallel bundle isomorphism
\[\Lambda^{1,1}T^\ast M\stackrel{\sim}{\longrightarrow}\Sym^+T^\ast M:\quad \alpha\longmapsto\alpha J\]
with the bundle of symmetric $2$-tensors commuting with $J$.

The bundle of $(1,1)$-forms is a homogeneous vector bundle with fiber $\Lambda^{1,1}\m\cong\u(\m)$. Via the infinitesimal isotropy representation, the Lie algebra $\k\cong\R\oplus\su(n_+)\oplus\su(n_-)$ embeds into $\Lambda^{1,1}\m$, thus defining a parallel splitting
\[\Lambda^{1,1}_0T^\ast M=\bbE_+\oplus\bbE_-\oplus\bbF\]
where $\bbE_\pm$ have fiber $\su(n_\pm)$ and $\bbF:=(\bbE_+\oplus\bbE_-)^\perp$. For $\alpha\in\Lambda^{1,1}T^\ast M$, we denote with $\alpha_\pm$ the projections to the subbundles $\bbE_\pm$, respectively.

\section{Fibrewise identities}
\label{sec:fibrewise}

In order to perform calculations on the complex Grassmannians, we need a bit of preparation. In particular, we will employ Casimir operators on representations of $K$, the eigenvalues of the curvature operator on $M$, commutator and anticommutator relations for endomorphisms of $TM$, and some auxiliary objects that shall be explained later.

\subsection{Casimir operators}
\label{sec:casimir}


We recall a few key facts about Casimir operators and record a few relations between different normalizations.

Let $\g$ be a compact real Lie algebra equipped with an invariant inner product $Q$. For any representation $\rho_\ast: \g\to\gl(V)$, its \emph{Casimir operator} (with respect to $Q$) is an equivariant endomorphism $\Cas^{\g,Q}_V\in\End_\g V$ defined as
\[\Cas^{\g,Q}_V:=-\sum_i\rho_\ast(X_i)^2,\]
where $(X_i)$ is any $Q$-orthonormal basis of $\g$.

If $V_\gamma$ is an irreducible representation of $\g$ of highest weight $\gamma\in\t^\ast$ (where $\t\subset\g$ is a maximal toral subalgebra), then the Casimir operator acts on $V_\gamma$ as multiplication with a constant $\Cas^{\g,Q}_\gamma$ given by \emph{Freudenthal's formula},
\begin{equation}
\Cas^{\g,Q}_\gamma=Q^\ast(\gamma,\gamma+2\delta_\g),
\label{freudenthal}
\end{equation}
where $Q^\ast$ is the dual of the restriction of $Q$ to $\t$, and $\delta_\g$ is the half-sum of positive roots of $\g$.

The \emph{Killing form} of a Lie algebra $\g$ is the invariant symmetric bilinear form $B_\g$ defined by
\[B_\g(X,Y):=\tr(\ad(X)\circ\ad(Y)),\qquad X,Y\in\g.\]
It is known to be nondegenerate if and only if $\g$ is semisimple, and negative-semidefinite if and only if $\g$ is compact. We assume from now on that $\g$ is compact and semisimple. Thus $-B_\g$ is an invariant inner product on $\g$ called the \emph{standard inner product}.

Clearly, Casimir operators behave under scaling of the inner product as
\begin{equation}
\Cas^{\g,cQ}_V=c^{-1}\Cas^{\g,Q}_V,\qquad c>0.
\label{casscale}
\end{equation}
The standard inner product is distinguished by the following normalization property.

\begin{prop}
\label{casadjoint}
On the adjoint representation of $\g$, $\Cas^{\g,-B_\g}_\g=\Id$.
\end{prop}

This fact may be utilized to determine the ratio between two invariant inner products that differ by a factor.

Let $G$ be a compact semisimple Lie group with Lie algebra $\g$, $K$ a closed subgroup with Lie algebra $\k$, and $\m$ an $\Ad(G)$-invariant complement of $\k$ in $\g$. Then $\m$ is canonically identified with the tangent space at the identity coset of the homogeneous space $M=G/K$ (and thus called the \emph{isotropy representation}), and after restriction to $\m$ the standard inner product $-B_\g$ induces a $G$-invariant Riemannian metric on $M$ called the \emph{standard metric}. The Einstein equation for this metric reduces to a simple condition on the Casimir operator of the isotropy representation \cite[Prop.~7.89, 7.92]{besse}.

\begin{prop}
The standard metric on $M=G/K$ is Einstein (with Einstein constant $E$) if and only if $\Cas^{\k,-B_\g}_\m=\lambda\Id$ for some constant $\lambda\in\R$. If this is the case, then
\[\lambda=2E-\frac12.\]
\end{prop}

This formula is usually employed to quickly find the Einstein constant of a standard homogeneous Einstein manifold. However, if $G/K$ is a symmetric space, then this is not needed, as the following proposition shows \cite[Prop.~7.93]{besse}.

\begin{prop}
\label{einsteinconst}
The standard metric on a symmetric space of compact type is Einstein with $E=\frac12$.
\end{prop}

At times it becomes necessary to work with different invariant inner products on the same Lie algebra. For example, if $\h\subset\g$ is a simple subalgebra, then necessarily $B_\h$ is a constant multiple of the restriction $B_\g\big|_\h$ since both are $\h$-invariant. Another example is the situation where $\ad: \k\to\so(\m)$ is faithful and $\h\subset\k$ is any subalgebra. Then the standard inner product $-B_\g\big|_\m$ induces on $\so(\m)\cong\Lambda^2\m$ an inner product $\langle\cdot,\cdot\rangle_{\Lambda^2}$ in the usual way, i.e.
\begin{equation}
\langle\alpha,\beta\rangle_{\Lambda^2}=\sum_{i<j}\alpha(e_i,e_j)\beta(e_i,e_j)=-\frac12\tr(\alpha\circ\beta),
\label{lam2inprod}
\end{equation}
where $(e_i)$ is a $-B_\g\big|_\m$-orthonormal basis of $\m$. Pulling this back along $\ad\big|_\h$ gives an invariant inner product on $\h$ that we shall also denote by $\langle\cdot,\cdot\rangle_{\Lambda^2}$.

For our purposes and in the context of the complex Grassmannians, both of these issues shall now be addressed.

\begin{prop}
\label{casCk}
On the defining representation $\C^k$ of $\su(k)$, $\Cas^{\su(k),-B_{\su(k)}}_{\C^k}=\frac{k^2-1}{2k^2}\Id$.
\end{prop}
\begin{proof}
We note first that the highest weight of the defining representation is the first fundamental weight $\omega_1$ in the convention of Bourbaki, while the highest weight of the adjoint representation is $\omega_1+\omega_{k-1}$. The Casimir eigenvalues of the fundamental weight representations are calculated for example in \cite{SW22} (see the table on p.~17). Taking into account the normalization in Prop.~\ref{casadjoint}, we find the desired value.
\end{proof}

\begin{lem}
\label{killingforms}
Under the standard inclusion $\su(k)\subset\su(n)$, $2\leq k\leq n-1$, one has
\[B_{\su(n)}\big|_{\su(k)}=\frac{n}{k}B_{\su(k)}.\]
\end{lem}
\begin{proof}
Let $c>0$ such that $B_{\su(n)}\big|_{\su(k)}=cB_{\su(k)}$. We consider the Casimir operator $\Cas^{\su(k),-B_{\su(n)}}_{\su(n)}$ and take a trace. On one hand,
\[\tr\Cas^{\su(k),-B_{\su(n)}}_{\su(n)}=-\sum_i\tr(\ad(X_i)^2)=-\sum_iB_{\su(n)}(X_i,X_i)=\dim\su(k)=k^2-1,\]
where $(X_i)$ is a $-B_{\su(n)}$-orthonormal basis of $\su(k)$. On the other hand, we note that the restriction of the adjoint representation of $\su(n)$ to the subalgebra $\su(k)$ splits as $\su(k)\oplus(n-k)\C^k$ plus an $(n-k)^2$-dimensional trivial summand. Thus
\begin{align*}
\tr\Cas^{\su(k),-B_{\su(n)}}_{\su(n)}&=c^{-1}\cdot\tr\Cas^{\su(k),-B_{\su(k)}}_{\su(k)}+(n-k)\tr\Cas^{\su(k),-B_{\su(k)}}_{\C^k}\\
&=c^{-1}\cdot\left((k^2-1)\cdot1+2k(n-k)\cdot\frac{k^2-1}{2k^2}\right)\\
&=c^{-1}\cdot(k^2-1)\cdot\frac{n}{k}
\end{align*}
by \eqref{casscale}, Prop.~\ref{casadjoint} and Prop.~\ref{casCk}. It follows that $c=\frac{n}{k}$.
\end{proof}

\begin{lem}
Under the standard inclusion of $\h_\pm:=\su(n_\pm)$ into $\k:=\mathfrak{s}(\u(n_+)\oplus\u(n_-))$,
\[\langle\cdot,\cdot\rangle_{\Lambda^2}\big|_{\h_\pm}=-\frac{n_\mp}{2n_\pm}B_{\h_\pm}.\]
\end{lem}
\begin{proof}
Let $(X_i)$ be an orthonormal basis of $\h_\pm$ with respect to $\langle\cdot,\cdot\rangle_{\Lambda^2}$, and $(e_j)$ a $-B_\g$-orthonormal basis of $\m$. We take a trace of $\Cas^{\h_\pm,\langle\cdot,\cdot\rangle_{\Lambda^2}}_\m$ as follows. Using the definition of the inner product on $2$-forms \eqref{lam2inprod}, we may calculate
\begin{align*}
\tr\Cas^{\h_\pm,\langle\cdot,\cdot\rangle_{\Lambda^2}}_\m&=\sum_{i,j}B_\g(\ad(X_i)^2e_j,e_j)=-\sum_{i,j}B_\g(\ad(X_i)e_j,\ad(X_i)e_j)\\
&=2\sum_i\langle X_i,X_i\rangle_{\Lambda^2}=2\dim\h_{\pm}=2(n_\pm^2-1).
\end{align*}
Let $c>0$ such that $\langle\cdot,\cdot\rangle_{\Lambda^2}\big|_{\h_\pm}=-cB_{\h_\pm}$. Recall that $\m\cong\C^{n_+}\otimes_\C(\C^{n_-})^\ast$ as a representation of $\k$. Restricted to $\h_\pm$ it splits as $\m\cong n_{\mp}\C^{n_\pm}$ (or its dual, which has the same Casimir eigenvalue). Hence, using \eqref{casscale} and Prop.~\ref{casCk}, we find
\begin{align*}
\tr\Cas^{\h_\pm,\langle\cdot,\cdot\rangle_{\Lambda^2}}_\m&=c^{-1}\cdot\tr\Cas^{\h_\pm,-B_{\h_\pm}}_\m=c^{-1}\cdot 2n_+n_-\cdot\frac{n_\pm^2-1}{2n_\pm^2}=c^{-1}\cdot(n_\pm^2-1)\cdot\frac{n_\mp}{n_\pm}.
\end{align*}
Thus $2c=\frac{n_\mp}{n_\pm}$ as claimed.
\end{proof}

\begin{kor}
\label{casvalues}
For the isotropy and adjoint representations of $\su(n_\pm)$,
\[\Cas^{\su(n_\pm),\langle\cdot,\cdot\rangle_{\Lambda^2}}_\m=\frac{n_\pm^2-1}{n_+n_-}\Id,\qquad \Cas^{\su(n_\pm),\langle\cdot,\cdot\rangle_{\Lambda^2}}_{\su(n_\pm)}=\frac{2n_\pm}{n_\mp}\Id.\]
\end{kor}

\subsection{The curvature operator}


On a Riemannian manifold $(M,g)$ with Levi-Civita connection $\nabla$, let $R$ be the Riemann curvature tensor defined by
\[R(X,Y)Z=[\nabla_X,\nabla_Y]Z-\nabla_{[X,Y]}Z,\qquad X,Y,Z\in\X(M).\]
Moreoever we denote by $\Ric(X,Y)=\tr(Z\mapsto R(Z,X)Y)$ the Ricci curvature and let
\[g(\Ric(X),Y)=\Ric(X,Y)\]
define the Ricci endomorphism $\Ric: T_pM\to T_pM$. The \emph{curvature operator (of the first kind)} is the self-adjoint linear operator $\curvop: \Lambda^2T_pM\to\Lambda^2T_pM$ derived from the curvature via
\[\langle\curvop(X\wedge Y),Z\wedge W\rangle=g(R(X,Y)Z,W),\]
where the inner product on $\Lambda^2T_pM$ is the usual one induced by $g$, cf.~\eqref{lam2inprod}. We note that in our convention, positive curvature operator implies \emph{negative} sectional curvature.

We will need some of the eigenvalues of the curvature operator on the complex Grassmannians. Although they are available in the literature, we take the liberty to rederive them using the so-called \emph{standard/Weitzenböck curvature endomorphism} $\curvendo$ which is given on any vector bundle $VM=\O(TM)\times_\rho V$ associated to the orthonormal frame bundle $\O(TM)$ as
\[\curvendo:=\sum_a\rho_\ast(\omega_a)\rho_\ast(\curvop\omega_a),\]
where $(\omega_a)$ is an orthonormal basis of $\so(T_pM)\cong\Lambda^2T_pM$, and $\rho_\ast$ is the differential of the representation $\rho: \O(n)\to\GL(V)$.

We shall make use of the following known facts.

\begin{prop}
\label{curvendo}
\begin{enumerate}[\upshape(i)]
 \item On $TM$, we have $\curvendo=\Ric$.
 \item On $\Lambda^2TM$, we have $\curvendo=\Ric_\ast+2\curvop$, where $\Ric_\ast$ is the extension of $\Ric$ as a derivation, i.e.
 \[\Ric_\ast(X\wedge Y)=\Ric(X)\wedge Y+X\wedge\Ric(Y).\]
\end{enumerate}
\end{prop}

\begin{prop}
\label{curvcas}
If $(M=G/K,g)$ is a Riemannian symmetric space with metric $g$ induced by the restriction to $\m$ of an invariant inner product $Q$ on $\g$, then
\[\curvendo=\Cas^{\k,Q}_V\]
on any homogeneous vector bundle $VM=G\times_KV$.
\end{prop}

The eigenvalues of $\curvop$ may now be calculated with ease.

\begin{lem}
\label{curveigen}
The eigenvalues of the curvature operator on the complex Grassmannians with the standard metric are given by
\[\curvop\omega=-\frac12\omega,\qquad\curvop\big|_{\bbE_\pm}=-\frac{n_\mp}{2n}\Id,\qquad\curvop\big|_{(\R\omega\oplus\bbE_+\oplus\bbE_-)^\perp}=0.\]
\end{lem}
\begin{proof}
Recall that $(M,g)$ is Einstein with $\Ric(X)=\frac{1}{2}X$ by Prop.~\ref{einsteinconst}. Thus we have $\Ric_\ast(X\wedge Y)=X\wedge Y$, and Prop.~\ref{curvendo} yields
\[\curvendo=\Id+2\curvop\]
on $\Lambda^2TM$. Since $\bbE_\pm$ are homogeneous vector bundles with fiber $\su(n_\pm)$, we may combine Prop.~\ref{curvcas} with Prop.~\ref{casadjoint} and Prop.~\ref{killingforms} to obtain
\[\curvop\big|_{\bbE_\pm}=\frac{1}{2}\Cas^{\k,-B_\g}_{\su(n_\pm)}-\frac{1}{2}\Id=\frac{n_\pm}{2n}\Cas^{\su(n_\pm),-B_{\su(n_\pm)}}_{\su(n_\pm)}-\frac{1}{2}\Id=\left(\frac{n_\pm}{2n}-\frac{1}{2}\right)\Id=-\frac{n_\mp}{2n}\Id.\]
On the trivial bundle $\R\omega$, the Casimir term vanishes, yielding $\curvop\omega=-\frac{1}{2}\omega$. For the sake of completeness we note that $\R\omega\oplus\bbE_+\oplus\bbE_-$ is fibrewise isomorphic to the Riemannian holonomy algebra $\k=\mathfrak{s}(\u(n_+)\oplus\u(n_-))$ of the symmetric space $(M,g)$, so the curvature operator vanishes on its orthogonal complement.
\end{proof}

\begin{bem}
Note that the eigenvalues of $\curvop$ have opposite sign as in \cite{NS23} due to different sign conventions for the curvature.
\end{bem}

\subsection{Commutators and anticommutators}
\label{sec:commanticomm}


The natural action of an endomorphism $A\in\End T_pM$ on tensors is as a derivation. On differential forms in particular it may be written as
\[A_\ast\alpha=\sum_iAe_i\wedge(e_i\intprod\alpha),\qquad\alpha\in\Lambda^\ast T_pM,\]
where $(e_i)$ is an orthonormal basis of $T_pM$. If $\alpha$ is a $2$-form, identified via the metric with a skew-symmetric endomorphism, then $A_\ast\alpha$ may be expressed using a commutator and an anticommutator as
\[A_\ast\alpha=[A_{\Lambda^2},\alpha]+\{A_{\Sym^2},\alpha\},\]
where $A_{\Lambda^2}$ and $A_{\Sym^2}$ are the skew-symmetric and symmetric part of $A$, respectively, according to the decomposition $\End TM\cong\Lambda^2TM\oplus\Sym^2TM$. This occurrence of both commutators and anticommutators in the natural action warrants a discussion about their behavior with respect to the relevant subbundles of $\Lambda^2TM$ on the complex Grassmannians.

Recall that the isotropy representation $\m$ of the complex Grassmannian $M=G/K$ is, as a complex representation, equivalent to $\C^{n_+}\otimes_\C(\C^{n_-})^\ast$. Under this equivalence, the almost complex structure $J$ translates to multiplication with $\i$. We may thus identify the (assocative) algebra $\End^+\m$ of endomorphisms commuting with $J$ with the algebra
\[\gl(n_+n_-,\C)=\gl(n_+,\C)\otimes\gl(n_-,\C)\]
acting on $\m\cong\C^{n_+}\otimes_\C(\C^{n_-})^\ast$ in the natural way, namely
\[(A\otimes B)(v\otimes\lambda)=(Av)\otimes(\lambda B^\ast)\]
for $A\in\gl(n_+,\C)$, $B\in\gl(n_-,\C)$, $v\in\C^{n_+}$, $\lambda\in(\C^{n_-})^\ast$. The subalgebras $\gl(n_\pm,\C)$, which embed using the Kronecker product into $\gl(n_+n_-,\C)$ as
\begin{align*}
 \gl(n_+,\C)&\hookrightarrow\gl(n_+n_-,\C):\qquad A\mapsto A\otimes I_{n_-},\\
 \gl(n_-,\C)&\hookrightarrow\gl(n_+n_-,\C):\qquad B\mapsto I_{n_+}\otimes B,
\end{align*}
clearly commute with each other. This seemingly trivial observation has important con\-se\-quen\-ces for commutation and anticommutation relations between elements of the vector bundles $\bbE_\pm$ and $\bbF$.

\begin{lem}
\label{comm}
$[\bbE_+,\bbE_-]=0$.
\end{lem}
\begin{proof}
Since $\bbE_\pm$ are subbundles of the algebra bundle $\End^+TM$ whose fibers correspond to the subspaces $\su(n_\pm)\subset\gl(n_\pm,\C)\subset\gl(n_+n_-,\C)\cong\End^+\m$, this follows directly from the observation above.
\end{proof}

\begin{lem}
\label{skewherm}
Under the above identification $\Lambda^{1,1}\m\cong\{\i AB\,|\,A\in\u(n_+),\ B\in\u(n_-)\}$. In particular
\[\bbF=\spann\{\alpha\beta J\,|\,\alpha\in\bbE_+,\ \beta\in\bbE_-\}.\]
\end{lem}
\begin{proof}
Identifying $\m\cong\C^{n_+}\otimes_\C(\C^{n_-})^\ast$, the space $\Lambda^{1,1}\m$ corresponds precisely to the subspace $\u(n_+n_-)\subset\gl(n_+n_-,\C)$. Following the above discussion, we may write
\[\gl(n_+n_-,\C)=\spann\{AB\,|\,A\in\gl(n_+,\C),\ B\in\gl(n_-,\C)\}.\]
Recall that the anticommutator of two skew-Hermitian matrices is Hermitian. Take now $A\in\u(n_+)$ and $B\in\u(n_-)$. Since they commute we have
\[\i AB=\frac{\i}{2}\{A,B\}\in\i\{\u(n_+n_-),\u(n_+n_-)\}\subset\u(n_+n_-)\cong\Lambda^{1,1}\m.\]
By counting dimensions, we see that these elements actually span $\Lambda^{1,1}\m$. Further splitting $\u(n_\pm)=\su(n_\pm)\oplus\i\R$ we obtain
\[\u(n_+n_-)=\i\R\oplus\su(n_+)\oplus\su(n_-)\oplus\spann\{iAB\,|\,A\in\su(n_+),\ B\in\su(n_-)\}.\]
Thus $\bbF$, being the invariant complement of $\R\omega\oplus\bbE_+\oplus\bbE_-$ in $\Lambda^{1,1} TM$, is associated to the last summand.
\end{proof}


\begin{kor}
\label{anticomm1}
If $\alpha,hJ\in\bbE_\pm$, then $\{h,\alpha\}\in\bbE_\pm\oplus\R\omega$.
\end{kor}
\begin{proof}
Indeed, $\{\u(n_\pm),\i\u(n_\pm)\}\subset\u(n_\pm)$, which is the fiber of $\bbE_\pm\oplus\R\omega$.
\end{proof}

\begin{kor}
\label{anticomm2}
If $\alpha\in\bbE_\pm$, $hJ\in\bbE_\mp$, then $\{h,\alpha\}\in\bbF$.
\end{kor}
\begin{proof}
Since $[\bbE_+,\bbE_-]=0$ by Lemma~\ref{comm} and both $\alpha$ and $h$ commute with $J$, we have $\{h,\alpha\}=2\alpha h$ which lies in $\bbF$ by Lemma~\ref{skewherm}.
\end{proof}

\subsection{Some auxiliary identities}


In order to parametrize the infinitesimal Einstein deformations and work out the second order obstruction to integrability, we introduce a few auxiliary objects and record some additional useful, albeit technical identities.

\begin{defn}
\label{defO}
Let $(\omega^\pm_a)$ always denote a local $\langle\cdot,\cdot\rangle_{\Lambda^2}$-orthonormal frame of $\bbE_\pm$.
\begin{enumerate}[(i)]
 \item We define the $4$-forms $\Omega_\pm:=\sum_a\omega^\pm_a\wedge\omega^\pm_a$ and $\Omega:=\frac{n_-^2-1}{2n}\Omega_+-\frac{n_+^2-1}{2n}\Omega_-$.
 \item For any $A\in\End TM$, we define $C_\pm(A):=\sum_a\omega^\pm_aA\omega^\pm_a$. Here we again identify $2$-forms with skew-symmetric endomorphisms.
\end{enumerate}
\end{defn}

\begin{bem}
\label{Oparallel}
It is not hard to see that the above definitions are independent of the choice of basis. The $4$-forms $\Omega_\pm$ serve the same purpose as the Kraines form in the quaternionic Kähler case $n_+=2$ (or $n_-=2$) \cite{NS23}. Since the bundles $\bbE_\pm$ are parallel and the $\omega^\pm_a$ are skew-symmetric, it follows that also the $\Omega_\pm$ and thus $\Omega$ are parallel.
\end{bem}

\begin{lem}
\label{Cendo}
The operators $C_\pm$ can be expressed as
\[C_\pm(A)=\frac{1}{2}\Cas^{\su(n_\pm),\langle\cdot,\cdot\rangle_{\Lambda^2}}_{\gl(\m)}(A)-\frac{n_\pm^2-1}{n_+n_-}A.\]
In particular $C_\pm$ is $\su(n_\pm)$- and thus $\k$-equivariant, i.e.
\[C_\pm([\alpha,A])=[\alpha,C_\pm(A)]\qquad\forall\alpha\in\R\omega\oplus\bbE_+\oplus\bbE_-,\]
and we have $C_\pm(A)=\frac{1}{n_+n_-}A$ if $A\in\bbE_\pm\oplus\bbF$ and $C_\pm(A)=-\frac{n_\pm^2-1}{n_+n_-}A$ if $A\in\R\omega\oplus\bbE_\mp$.
\end{lem}
\begin{proof}
As $\su(n_\pm)\subset\so(\m)$ acts on $\gl(\m)$ by the commutator, we calculate
\begin{align*}
 \Cas^{\su(n_\pm),\langle\cdot,\cdot\rangle_{\Lambda^2}}_{\gl(\m)}(A)&=-\sum_a[\omega^\pm_a,[\omega^\pm_a,A]]=\sum_a(2\omega^\pm_aA\omega^\pm_a-A(\omega^\pm_a)^2-(\omega^\pm_a)^2A)\\
 &=2C_\pm(A)+\{A,\Cas^{\su(n_\pm),\langle\cdot,\cdot\rangle_{\Lambda^2}}_{\m}\}.
\end{align*}
Corollary~\ref{casvalues} now implies the first assertion. The equivariance under $\k$ follows from the fact that $\Cas^{\su(n_\pm)}$ commutes with the action of $\su(n_\pm)$ and that $[\bbE_+,\bbE_-]=0=[\bbE_\pm,J]$.

Restricting the $\k$-module
\[\Lambda^{1,1}\m\cong\R\oplus\su(n_+)\oplus\su(n_-)\oplus\su(n_+)\otimes\su(n_-)\]
to $\su(n_\pm)$, we observe that $\su(n_\pm)$ acts trivially on $\R\oplus\su(n_\mp)$, while the rest is equivalent to $n_\mp^2$ copies of the adjoint representation of $\su(n_\pm)$. By Corollary~\ref{casvalues} we obtain
\begin{align*}
 C_\pm(A)&=-\frac{n_\pm^2-1}{n_+n_-}A,&A&\in\R\omega\oplus\bbE_\mp,\\
 C_\pm(A)&=\left(\frac{1}{2}\cdot\frac{2n_\pm}{n_\mp}-\frac{n_\pm^2-1}{n_+n_-}\right)A=\frac{1}{n_+n_-}A,&A&\in\bbE_\pm\oplus\bbF.
\end{align*}
\end{proof}

The following lemma is elementary but stated for convenience.

\begin{lem}
\label{wedgein}
For any $\alpha,\beta,\gamma\in\Lambda^2TM$, we have
\[\alpha\intprod(\beta\wedge\gamma)=\langle\alpha,\beta\rangle\gamma+\langle\alpha,\gamma\rangle\beta+\beta\alpha\gamma+\gamma\alpha\beta.\]
In particular
\[\alpha\intprod(\beta\wedge\beta)=2\langle\alpha,\beta\rangle\beta+2\beta\alpha\beta.\]
\end{lem}

We shall now take a closer look at the $4$-forms $\Omega_\pm$ and $\Omega$.

\begin{lem}
\label{XinOp}
For any $X\in TM$,
\begin{enumerate}[\upshape(i)]
 \item $\sum_ie_i\wedge(e_i\wedge X)_\pm=-\frac{1}{2}X\intprod\Omega_\pm$,
 \item $\sum_ie_i\intprod(e_i\wedge X)_\pm=\frac{n^2_\pm-1}{n_+n_-}X$.
\end{enumerate}
\end{lem}
\begin{proof}
\begin{enumerate}[(i)]
 \item We calculate
 \begin{align*}
 X\intprod\Omega_\pm&=2\sum_a(X\intprod\omega^\pm_a)\wedge\omega^\pm_a=2\sum_{i,a}\langle e_i,X\intprod\omega^\pm_a\rangle e_i\wedge\omega^\pm_a\\
 &=2\sum_{i,a}\langle X\wedge e_i,\omega^\pm_a\rangle e_i\wedge\omega^\pm_a=-2\sum_ie_i\wedge(e_i\wedge X)_\pm.
 \end{align*}
 \item Similarly,
 \begin{align*}
 \sum_ie_i\intprod(e_i\wedge X)_\pm&=-\sum_{i,a}\langle X\wedge e_i,\omega^\pm_a\rangle\omega^\pm_a(e_i)=-\sum_{i,a}\langle e_i,X\intprod\omega^\pm_a\rangle\omega^\pm_a(e_i)\\
 &=-\sum_a(\omega^\pm_a)^2(X)=\Cas^{\su(n_\pm),\langle\cdot,\cdot\rangle_{\Lambda^2}}_\m(X),
 \end{align*}
 and the assertion now follows from Corollary~\ref{casvalues}.
\end{enumerate}
\end{proof}

\begin{lem}
\label{inOp}
For any $\alpha\in\Lambda^2TM$, we have We note $\alpha\intprod\Omega_\pm=2\alpha_\pm+2C_\pm(\alpha)$.
\end{lem}
\begin{proof}
It follows from Lemma~\ref{wedgein} that
\begin{align*}
\alpha\intprod\Omega_\pm&=\sum_a\alpha\intprod(\omega_a^\pm\wedge\omega_a^\pm)=\sum_a(2\langle\alpha,\omega_a^\pm\rangle\omega_a^\pm+2\omega_a^\pm\alpha\omega_a^\pm)=2\alpha_\pm+2C_\pm(\alpha).
\end{align*}
\end{proof}

\begin{lem}
\label{Oprim}
The $4$-form $\Omega$ is primitive, i.e.~$\omega\intprod\Omega=0$.
\end{lem}
\begin{proof}
Since $\omega\perp\bbE_p$, Lemma~\ref{Cendo} and Lemma~\ref{inOp} imply that
\[\omega\intprod\Omega_\pm=2C_\pm(\omega)=-2\frac{n_\pm^2-1}{n_+n_-}\omega\]
and one checks using the definition of $\Omega$ that $\omega\intprod\Omega=0$.
\end{proof}

\begin{kor}
\label{Onondeg}
The $4$-forms $\Omega_\pm$ and $\Omega$ are all nondegenerate in the sense that if $X\intprod\Omega_\pm=0$ or $X\intprod\Omega=0$ for any $X\in TM$, then $X=0$.
\end{kor}
\begin{proof}
This is a consequence of the holonomy principle: indeed, the kernels of the maps $X\mapsto X\intprod\Omega_\pm$, $X\mapsto X\intprod\Omega$ would define parallel subbundles of $TM$. Since $(M,g)$ has irreducible holonomy, it follows that these $4$-forms are either nondegenerate or vanish. However Lemma~\ref{Cendo} and Lemma~\ref{inOp} imply that
\[\alpha\intprod\Omega_\pm=\frac{n_+n_-+1}{n_+n_-}\alpha\]
for any $\alpha\in\bbE_\pm$, so $\Omega_\pm$ cannot vanish. A similar computation shows that the same is true for the linear combination $\Omega$.
\end{proof}

\section{Killing fields and differential identities}
\label{sec:diff}

 
We recall that a Killing vector field $X$ on a Kähler manifold $(M, g, J)$ is Hamiltonian with respect to the Kähler form $\omega$, i.e. $L_X\omega = 0$. Hence, if $M$ is compact and simply connected, we have a globally defined function $z_X$, called Killing potential uniquely determined from  $X \intprod \omega = dz_X$ and $\int_M z_X \vol = 0$. Note that the equation $L_XJ=0$ also implies that $dX$ is in $\Omega^{1,1}(M)$.

Assume that $g$ is a Kähler--Einstein metric of Einstein constant $\frac12$. Then Killing vector fields and their Killing potentials are both in the kernel of $\Delta - \Id$, where $\Delta=d^\ast d+dd^\ast$ is the Laplacian defined by the metric $g$. In fact, by Matsushima's Theorem, the map $X \mapsto z_X$ defines an isomorphism between the space of  Killing vector fields $\isom(M,g)$ and the (minimal) eigenspace $\ker (\Delta - \Id)$ on functions. In particular, we have $\Ltwoinprod{X}{Y} =  \Ltwoinprod{z_X}{z_Y}$ and we see that if a Killing potential is $L^2$-orthogonal to all other Killing potentials, it has to vanish.

By a theorem of Kostant \cite[Thm.~3.3]{kostant}, if $X$ is a Killing vector field on a compact manifold $(M,g)$, the skew-symmetric endomorphism $dX$ is contained in the Riemannian holonomy algebra of $g$ at each point. Hence, for the complex Grassmannians, $dX$ takes values in the subbundle $\R\omega\oplus\bbE_+\oplus\bbE_-$ with fiber $\k\subset\Lambda^{1,1}\m$.

\begin{bem}
 As mentioned in \S\ref{sec:grassmann}, the Killing vector fields on the complex Grassmannian $M=G/K=\SU(n)/\rmS(\U(n_+)\times\U(n_-))$ are precisely the fundamental vector fields generated by the action of $G=\SU(n)$. For such a fundamental vector field $\tilde X$ generated by $X\in\g$, clearly its value at the identity coset only depends on its $\m$-part -- on the other hand, its exterior derivative is (at the identity coset) given by
\[(d\tilde X)_o(Y,Z)=g([X,Y],Z),\qquad Y,Z\in\m\cong T_oM,\]
cf.~\cite[Lem.~7.27]{besse}. By the Cartan relation $[\m,\m]\subset\k$, we see that this only depends on the $\k$-part of the Lie algebra element $X$, which is consistent with the conclusion above.
\end{bem}

\begin{lem}
\label{dXprim}
If $X$ is a Killing vector field,
\[(dX)_\prim=(dX)_++(dX)_-=dX-\frac{z_X}{n_+n_-}\omega.\]
\end{lem}
\begin{proof}
Since $dX$ takes values in $\R\omega\oplus\bbE_+\oplus\bbE_-$, we know that
\[dX=(dX)_++(dX)_-+\frac{\langle dX,\omega\rangle}{|\omega|^2}\omega.\]
We have $|\omega|^2=-\frac{1}{2}\tr(J^2)=\dim_\C M=n_+n_-$. Recall that $X$ being Killing is equivalent to $dX=2\nabla X$. Using that $J$ and $g$ are parallel, we may thus calculate
\begin{align*}
\langle dX,\omega\rangle&=\frac{1}{2}\sum_i\langle Je_i,e_i\intprod dX\rangle=-\sum_i\langle e_i,\nabla_{e_i}(JX)\rangle\\
&=d^\ast JX=d^\ast dz_X=z_X,
\end{align*}
and the assertion follows.
\end{proof}

\begin{lem}
\label{ddX}
Any Killing vector field $X$ satisfies the following identities:
\begin{enumerate}[\upshape(i)]
 \item $\displaystyle d(dX)_\pm=\frac{n_\mp}{2n}X\intprod\Omega_\pm$,
 \item $\displaystyle d^\ast(dX)_\pm=\frac{n_\pm^2-1}{n_\pm n}X$.
 \item $\displaystyle d|(dX)_\pm|^2=\frac{2n_\mp}{n}X\intprod(dX)_\pm$.
 \item $\displaystyle \Delta|(dX)_\pm|^2=\frac{2n_\mp}{n}|(dX)_\pm|^2-\frac{2(n_\pm^2-1)n_\mp}{n_\pm n^2}|X|^2$.
\end{enumerate}
\end{lem}
\begin{proof}
\begin{enumerate}[(i)]
\item We recall that since $X$ is Killing, it satisfies $\nabla dX=2R(\cdot,X)$. Using that the bundles $\bbE_\pm$ are parallel and that the curvature operator $\curvop$ is symmetric and has eigenvalues $-\frac{n_\mp}{2n}$ on $\bbE_\pm$ by Lemma~\ref{curveigen}, we calculate
\begin{align*}
 d(dX)_\pm&=\sum_ie_i\wedge\nabla_{e_i}(dX)_\pm=\sum_ie_i\wedge(\nabla_{e_i}dX)_\pm\\
 &=2\sum_ie_i\wedge R(e_i,X)_\pm=-\frac{n_\mp}{n}\sum_ie_i\wedge(e_i\wedge X)_\pm.
\end{align*}
Applying Lemma~\ref{XinOp} (i) now yields the first assertion.
\item Similar to the above, we find
\begin{align*}
 d^\ast(dX)_\pm&=-\sum_ie_i\intprod\nabla_{e_i}(dX)_\pm=\frac{n_\mp}{n}\sum_ie_i\intprod(e_i\wedge X)_\pm.
\end{align*}
The assertion follows now from Lemma~\ref{XinOp} (ii).
\item Using the same identities as above, we calculate
\begin{align*}
  d|(dX)_\pm|^2&=-4\sum_i\langle R(X,e_i)_\pm,(dX)_\pm\rangle e_i=\frac{2n_\mp}{n}X\intprod(dX)_\pm.
 \end{align*}
\item By taking the codifferential in (iii) and applying Lemma~\ref{XinOp} (ii), we obtain
 \begin{align*}
  \Delta|(dX)_\pm|^2&=\frac{2n_\mp}{n}\sum_i\left((dX)_\pm(e_i,\nabla_{e_i}X)+2\langle R(X,e_i)_\pm,X\wedge e_i\rangle\right)\\
  &=\frac{2n_\mp}{n}|(dX)_\pm|^2-\frac{2(n_\pm^2-1)n_\mp}{n_\pm n^2}|X|^2.
 \end{align*}
\end{enumerate}
\end{proof}

\section{The infinitesimal Einstein deformations}
\label{sec:ied}


It has been known for some time that the space $\varepsilon(g)$ of infinitesimal Einstein de\-for\-ma\-tions of the complex Grassmannian is, as a $G$-submodule of the space of tt-tensors, isomorphic to $\g$ and thus to the space $\isom(M,g)$ of Killing vector fields (\cite{Koiso80}, see also \cite[Prop.~8.4]{GG04}). The infinitesimal Einstein deformations have been explicitly described in the case $n_+=n_-$ by Gasqui--Goldschmidt \cite[\S VIII.4]{GG04}, while a parametrization by Killing vector fields has been given in the case $n_-=2$ by Nagy--Semmelmann \cite[Prop.~6.2]{NS23}. We extend the latter approach to all $n_+,n_-\geq2$.

We shall parametrize the infinitesimal Einstein deformations of $(M,g)$ in terms of Killing vector fields using the following map.

\begin{defn}
\label{parameid}
Let $e: \X(M)\to\Omega^{1,1}_\prim(M)$ be the linear map defined by
\[X\longmapsto e_X:=\frac{n_-^2-1}{n_-}(dX)_+-\frac{n_+^2-1}{n_+}(dX)_-.\]
\end{defn}

\begin{satz}
\label{deX}
If $X$ is any Killing vector field, then
\begin{enumerate}[\upshape(i)]
 \item $de_X=X\intprod\Omega$,
 \item $d^\ast e_X=0$,
 \item $\Delta e_X=e_X$.
\end{enumerate}
Moreover, the map $e$ is injective restricted to Killing vector fields.
\end{satz}
\begin{proof}
(i) and (ii) follow immediately from Lemma~\ref{ddX} and the definition of $\Omega$ (Defi\-ni\-tion~\ref{defO}). Moreover, by (i) and (ii), we find that $\Delta e_X=d^\ast(X\intprod\Omega)$. Next, we observe
\[d^\ast(X\intprod\Omega)=dX\intprod\Omega-X\intprod d^\ast\Omega,\]
and $d^\ast\Omega=0$ since $\Omega$ is parallel (Remark~\ref{Oparallel}). Now combining Lemma~\ref{Cendo}, Lemma~\ref{inOp} and the fact that $\Omega$ is primitive (Lemma~\ref{Oprim}), it is straightforward to calculate that
\begin{align*}
 \Delta e_X&=dX\intprod\Omega=(dX)_+\intprod\Omega+(dX)_-\intprod\Omega\\
 &=\frac{n_-^2-1}{n_-}(dX)_+-\frac{n_+^2-1}{n_+}(dX)_-=e_X.
\end{align*}
Finally, let $X$ be a Killing field such that $e_X=0$. Then by (i), $X\intprod\Omega=0$, and the nondegeneracy of $\Omega$  (Corollary~\ref{Onondeg}) implies $X=0$.
\end{proof}

The infinitesimal Einstein deformations of a Kähler--Einstein metric were described by Koiso \cite[Prop.~7.3]{Koiso83}, cf.~\cite[\S5]{NS23}. Generally $\varepsilon(g)$ splits into two subspaces $\varepsilon^\pm(g)$ of symmetric $2$-tensors commuting resp.~anticommuting with $J$, and we have
\[\varepsilon^+(g)=\{\alpha J\,|\,\alpha\in\Omega^{1,1}_0(M),\quad d^\ast\alpha=0,\quad \Delta\alpha=2E\alpha\}.\]
For the complex Grassmannians we have thus constructed an injective map from the space of Killing fields to $\varepsilon^+(g)J$. By the fact that $\isom(M,g)$ has the same dimension as $\varepsilon(g)$, we arrive at the following statement.

\begin{kor}
\label{killingparam}
The map $X\mapsto e_XJ$ is a linear isomorphism $\isom(M,g)\stackrel{\sim}{\to}\varepsilon^+(g)=\varepsilon(g)$.
\end{kor}

\begin{bem}
One may also arrive at the desired isomorphism by a purely algebraic argument. In light of the Peter--Weyl Theorem for homogeneous vector bundles \cite[Thm.~5.3.6]{Wa}, which states that given such a bundle $VM=G\times_KV$ over a homogeneous space $M=G/K$ of a compact Lie group $G$,
\[L^2(M,VM)\cong\closedsum_{\gamma\in\hat G}V_\gamma\otimes\Hom_K(V_\gamma,V)\]
as a $G$-module (where the sum ranges over all equivalence classes of irreducible repre\-sen\-ta\-tions of $G$), everything equivariantly derived from Killing vector fields plays out (in our case) on the Fourier mode $V_\gamma=\g$, the adjoint representation.

The infinitesimal Einstein deformations fit into an exact sequence
\[0\longrightarrow\isom(M,g)\longrightarrow\ker(\Delta-2E)\big|_{\Omega^1(M)}\longrightarrow\ker(\LL-2E)\big|_{\Sy^2_0(M)}\longrightarrow\varepsilon(g)\longrightarrow0,\]
cf.~\cite[\S4]{S22}. Since on an irreducible compact symmetric space $\g\cong\isom(M,g)$ and $\g$ is the only Fourier mode with Laplace eigenvalue $2E$ \cite[Lem.~5.2]{Koiso82}, this sequence descends to the level of Fourier coefficients as
\[0\longrightarrow\Hom_K(\g,\m)\longrightarrow\Hom_K(\g,\m)\longrightarrow\Hom_K(\g,\Sym^2_0\m)\longrightarrow T\longrightarrow0.\]
For the complex Grassmannians, we have $\Hom_K(\g,\m)=\spann\{\pr_\m\}$, $\Hom_K(\g,\Sym^2_0\m)$ is via $J$ isomorphic to $\Hom_K(\g,\Lambda^{1,1}_\prim\m)$ which in turn is spanned by both $\ad_\k\circ\pr_{\su(n_\pm)}$, and $T$ is a suitable one-dimensional subspace of $\Hom_K(\g,\Sym^2_0\m)$ such that the cor\-res\-pon\-ding symmetric $2$-tensors are divergence-free. That is, $T$ is the space of Fourier coefficients corresponding to the $G$-submodule $\varepsilon(g)$.

A straightforward calculation, also using Casimir operators, yields the appropriate linear combination of $\ad_\k\circ\pr_{\su(n_\pm)}$ (or $(dX)_\pm$ in Defi\-ni\-tion~\ref{parameid}, respectively) that spans the one-dimensional subspace $JT$ inside $\Hom_K(\g,\Lambda^{1,1}_\prim\m)$.
\end{bem}

\section{Invariant cubic forms}
\label{sec:invforms}

As explained in the introduction, the second order integrability obstruction manifests as an invariant cubic form on $\varepsilon(g)$. Ultimately our goal is to reduce it to an invariant polynomial on the Lie algebra $\g=\su(n)$ and in particular to show that it does not vanish identically. Throughout the calculation in \S\ref{sec:obstruction} we will encounter various such polynomials, however first we shall record a few of their properties and relations between them.

\subsection{Polynomials on $\su(n)$}


The following fact is well-known, cf.~\cite[Prop.~2.1]{GG04}, and has already been exploited in order to study the second order obstruction on homogeneous spaces of $\SU(n)$ \cite{BHMW,NS23,coindex}.

\begin{prop}
\label{invcubic}
The space $(\Sym^3\g^\ast)^G$ of invariant cubic forms on $\g=\su(n)$ is one-dimensional and spanned by the element $P_0$, defined by
\[P_0(X):=\i\tr(X^3),\qquad X\in\g.\]
\end{prop}

\begin{bem}
Recall that for any homogeneous polynomial, the associated symmetric multilinear form may be recovered via polarization. Using the same symbol for the multilinear form, we may thus write
\[P_0(X,Y,Z)=\frac{\i}{2}\tr(XYZ+ZYX),\qquad X,Y,Z\in\g.\]
Moreover, $P_0\in\Sym^3\g^\ast$ is contained in $\Sym^2\g^\ast\otimes\g^\ast$ as the $\g^\ast$-valued quadratic form $X\mapsto P_0(X,X,\cdot)$.
\end{bem}

\begin{defn}
Let $\quadric\subset\g$ denote the variety defined as the zero set of the quadratic map $X\mapsto P_0(X,X,\cdot)$, i.e.
\[\quadric:=\{X\in\g\,|\,P_0(X,X,Y)=0\ \forall Y\in\g\}.\]
\end{defn}

This variety may be described explicitly, cf.~\cite[Lem.~3.3]{BHMW}, \cite[Lem.6.7]{NS23}.

\begin{prop}
\label{variety}
\begin{enumerate}[\upshape(i)]
 \item $\quadric=\{X\in\su(n)\,|\,X^2=\frac{\tr(X^2)}{n}I_n\}$.
 \item If $n$ is odd, then $\quadric=\{0\}$.
 \item If $n=2k$ is even, then $\quadric$ is the union of $\SU(n)$-orbits of all block matrices of the form
 \[\begin{pmatrix}
    \i t I_{k}&0\\
    0&-\i t I_{k}
   \end{pmatrix},\qquad t\in\R,
\]
and thus a cone over the complex Grassmannian $\SU(2k)/\rmS(\U(k)\times\U(k))$.
\end{enumerate}
\end{prop}

\subsection{Polynomials on Killing vector fields}
\label{sec:killingpol}


We define the following $G$-invariant cubic forms on $\isom(M,g)\cong\g$.

\begin{defn}
For any Killing vector field $X$, let
\begin{enumerate}[(i)]
 \item $\displaystyle\mu(X):=\int_Mz_X^3\vol$,
 \item $\displaystyle\nu_\pm(X):=\int_M|(dX)_\pm|^2 z_X\vol$.
\end{enumerate}
\end{defn}

\begin{bem}
The $G$-invariance of $\mu,\nu_\pm$ is clear from the fact that $G$ acts by translation and the mappings $X\mapsto z_X$, $X\mapsto(dX)_\pm$ are all $G$-equivariant. Thus integrating over $M$ with respect to the invariant volume form $\vol$ averages out the action of $G$.

It was shown in \cite{HMW} that the polynomial $\mu$ does not vanish identically on $\isom(M,g)$ provided $n_+\neq n_-$. However for $n_+=n_-$ it does vanish, as we will shortly see.
\end{bem}

\begin{bem}
\label{polar}
 The (partially) polarized forms of $\mu$ and $\nu_\pm$ are given by
 \[\mu(X,X,Y)=\int_Mz_X^2z_Y\vol,\qquad\nu_\pm(X,X,Y)=\int_M|(dX)_\pm|^2z_Y\vol.\]
 The latter is not obvious, but follows along the lines of \cite[Lem.~6.8]{NS23}.
\end{bem}

Since $\mu,\nu_\pm$ are both $G$-invariant, they are (after application of the isomorphism $\g\cong\isom(M,g)$) multiples of the polynomial $P_0$ described above. Unfortunately, we cannot leave it at this observation. In order to obtain an expression for the obstruction polynomial $\Psi$ as a linear combination of $\mu$ and $\nu_\pm$ and to show that $\Psi$ corresponds to a nonzero multiple of $P_0$ defined above, it is necessary to relate the polynomials $\mu,\nu_\pm$ to other cubic integrals. In particular we need the following nonvanishing result, which shall be proved in \S\ref{sec:nu}.

\begin{lem}
\label{nu}
The polynomials $\nu_\pm$ do not vanish identically.
\end{lem}

First, however, we shall discuss a few relations between the polynomials $\mu,\nu_\pm$ and some cubic integrals of Killing vector fields.

\subsection{Relations between cubic integrals}
\label{sec:cubicint}

To avoid redundancy, we note here that the proofs in this section repeatedly use the following facts: any Killing field $X$ satisfies $dX=2\nabla X$ and $\nabla dX=2R(\cdot,X)$, while its Killing potential satisfies $dz_X=JX$ as well as $\Delta z_X=z_X$. Moreoever, as a consequence of Lemma~\ref{curveigen} and the symmetry of $\curvop$, we have the identity $R(X,Y)_\pm=-\frac{n_\mp}{2n}(X\wedge Y)_\pm$ for all $X,Y\in T_pM$, and thus, since the bundles $\bbE_\pm$ are parallel,
\[\nabla_Y(dX)_\pm=\frac{n_\mp}{n}(X\wedge Y)_\pm\]
for all $Y\in TM$ and Killing vector fields $X$. We also remind the reader that $n=n_++n_-$.

\begin{lem}
\label{cubic1}
For any Killing vector fields $X$ and $Y$,
\begin{enumerate}[\upshape(i)]
 \item $\displaystyle\int_M|X|^2z_Y\vol=\frac12\int_Mz_X^2z_Y\vol$,
 \item $\displaystyle\int_M|dX|^2z_X\vol=0$.
\end{enumerate}
\end{lem}
\begin{proof}
\begin{enumerate}[\upshape(i)]
 \item Using that $dz_X=JX$ and $\Delta z_X=z_X$, we find
 \begin{align*}
  d^\ast(z_X^2dz_Y)&=-2\langle JX,JY\rangle z_X+z_X^2d^\ast dz_Y=-2\langle X,Y\rangle z_X+z_X^2z_Y,\\
  d^\ast(z_Xz_Ydz_X)&=-|JX|^2z_X-\langle JX,JY\rangle z_X+z_Xz_Yd^\ast dz_X\\
  &=-|X|^2z_X-\langle X,Y\rangle z_X+z_X^2z_Y.
 \end{align*}
 Combining the above and integrating yields the assertion.
 \item First, integrating by parts we obtain
 \[\int_M(\Delta|X|^2)z_X\vol=\int_M|X|^2\Delta z_X\vol=\int_M|X|^2z_X\vol\]
 Using a local orthonormal basis that is parallel at the point of evaluation, we calculate
 \begin{align*}
  \Delta|X|^2&=-2\sum_ie_i(\langle\nabla_{e_i}X,X\rangle)=\sum_ie_i(\langle X\wedge e_i,dX\rangle)\\
  &=-\sum_i\langle e_i\wedge\nabla_{e_i}X,dX\rangle-\sum_i\langle e_i\wedge X,\nabla_{e_i}dX\rangle\\
  &=-\langle dX,dX\rangle+2\Ric(X,X)=-|dX|+|X|^2.
 \end{align*}
 Thus we also obtain
 \[\int_M(\Delta|X|^2)z_X\vol=-\int_M|dX|^2z_X\vol+\int_M|X|^2z_X\vol\]
 and conclude that $\int_M|dX|^2z_X\vol=0$.
\end{enumerate}
\end{proof}

\begin{bem}
In fact, the statement of Lemma~\ref{cubic1} holds true on any compact Kähler--Einstein manifold with $E=\frac{1}{2}$.
\end{bem}

\begin{kor}
\label{munu}
The polynomials $\mu,\nu_\pm$ satisfy the following relations.
\begin{enumerate}[\upshape(i)]
 \item $\displaystyle\nu_++\nu_-+\frac{1}{n_+n_-}\mu=0$.
 \item $\displaystyle(n_\mp-n_\pm)\nu_\pm=\frac{n_\mp(n_\pm^2-1)}{n_\pm n}\mu$. In particular $\mu=0$ if $n_+=n_-$.
\end{enumerate}
\end{kor}
\begin{proof}
\begin{enumerate}[(i)]
 \item This follows directly from combining Lemma~\ref{dXprim} with Lemma~\ref{cubic1} (i).
 \item We recall from Lemma~\ref{ddX} (iv) that
 \begin{align*}
  \Delta|(dX)_\pm|^2&=\frac{2n_\mp}{n}|(dX)_\pm|^2-\frac{2(n_\pm^2-1)n_\mp}{n_\pm n^2}|X|^2.
 \end{align*}
 Taking the product with $z_X$ and integrating, we find with Lemma~\ref{cubic1} (i) that
 \[\int_M(\Delta|(dX)_\pm|^2)z_X\vol=\frac{2n_\mp}{n}\nu_\pm(X)-\frac{(n_\pm^2-1)n_\mp}{n_\pm n^2}\mu(X).\]
 But on the other hand, integration by parts simply yields
 \[\int_M(\Delta|(dX)_\pm|^2)z_X\vol=\int_M|(dX)_\pm|^2\Delta z_X\vol=\nu_\pm(X)\]
 Together, these imply the assertion.
\end{enumerate}
\end{proof}

\begin{lem}
\label{cubic2}
For any Killing vector field $X$,
\begin{enumerate}[\upshape(i)]
 \item $\displaystyle\int_M(dX)_\pm(JX,X)\vol=-\nu_\pm(X)+\frac{n_\pm^2-1}{2n_\pm n}\mu(X)$,
 \item $\displaystyle\int_M\langle(dX)_+^2J,(dX)_\pm\rangle\vol=\frac{-n_\pm^2n_\mp+n_\pm+3n_\mp}{n_+n_-n}\nu_\pm(X)+\frac{(n_\pm^2-1)(n_\pm^2-2)}{2n_\pm^2n^2}\mu(X)$.
\end{enumerate}
\end{lem}
\begin{proof}
\begin{enumerate}[(i)]
 \item First, we calculate
 \begin{align*}
  \nabla_Y(z_XX\intprod(dX)_\pm)&=\langle JX,Y\rangle X\intprod(dX)_\pm+\frac{1}{2}z_XdX(Y)\intprod(dX)_\pm\\
  &\phantom{=}+\frac{n_\mp}{n}z_XX\intprod(X\wedge Y)_\pm.
 \end{align*}
 Together with Lemma~\ref{XinOp} (ii) it follows that
 \begin{align*}
  d^\ast(z_XX\intprod(dX)_\pm)&=(dX)_\pm(JX,X)+|(dX)_\pm|^2 z_X-\frac{n_\pm^2-1}{n_\pm n}|X|^2z_X
 \end{align*}
 and by integrating and combining with Lemma~\ref{cubic1} (i) we obtain the desired equality.
 \item We start by calculating the codifferential of the vector field $(dX)_\pm^2JX$. Similar to the above, we find that
 \begin{align*}
  \nabla_Y((dX)_\pm^2JX)&=\frac{n_\mp}{n}\{(X\wedge Y)_\pm,(dX)_\pm\}JX+\frac{1}{2}(dX)_\pm^2JdX(Y).
 \end{align*}
 Let us consider the terms occurring in $d^\ast(dX)_\pm^2JX$ one by one. First,
 \begin{align*}
  \sum_i\langle e_i,(X\wedge e_i)_\pm(dX)_\pm JX\rangle&=\sum_i\langle e_i\intprod(e_i\wedge X)_\pm,(dX)_\pm JX\rangle\\
  &=\frac{n_\pm^2-1}{n_+n_-}(dX)_\pm(JX,X)
 \end{align*}
 by Lemma~\ref{XinOp} (ii). Second,
 \begin{align*}
  \sum_i\langle e_i,(dX)_\pm(X\wedge e_i)_\pm JX\rangle&=\sum_{i,a}\langle e_i,X\intprod\omega_a^\pm\rangle (dX)_\pm(JX\intprod\omega_a^\pm,e_i)\\
  &=\sum_a(dX)_\pm(JX\intprod\omega_a^\pm,X\intprod\omega_a^\pm)\\
  &=-C_\pm((dX)_\pm)(JX,X)=-\frac{1}{n_+n_-}(dX)_\pm(JX,X)
 \end{align*}
 by virtue of Lemma~\ref{Cendo}. Third, we note that by Corollary~\ref{anticomm1} applied to $\alpha=hJ=(dX)_\pm$, the endomorphism $(dX)_\pm^2J$ is a section of $\bbE_\pm\oplus\R\omega$. Thus, using Lemma~\ref{dXprim}, we find that
 \begin{align*}
  \sum_i\langle e_i,(dX)_\pm^2JdX(e_i)\rangle&=-2\langle (dX)_\pm^2J,dX\rangle=-2\langle (dX)_\pm^2J,dX_\pm\rangle+\frac{2z_X}{n_+n_-}|(dX)_\pm|^2.
 \end{align*}
 Note that the factor $-2$ comes from the definition of the inner product on $2$-forms, cf.~\eqref{lam2inprod}. Putting everything together, we obtain
 \begin{align*}
  d^\ast((dX)^2_\pm JX)&=-\frac{n_\pm^2-2}{n_\pm n}(dX)_\pm(JX,X)+\langle (dX)_\pm^2J,(dX)_\pm\rangle-\frac{z_X}{n_+n_-}|(dX)_\pm|^2
 \end{align*}
 and after integrating and applying (i) the assertion follows.
\end{enumerate}
\end{proof}

\subsection{The nonvanishing of $\nu_\pm$}
\label{sec:nu}

The purpose of this section is to prove Lemma~\ref{nu}. We note that in the case $n_+\neq n_-$, the nonvanishing of $\nu_\pm$ follows from Corollary~\ref{munu} (ii) and the nonvanishing of $\mu$ \cite[\S5]{HMW}. Thus it remains to focus on the case $n_+=n_-$.

For any Killing vector field $X$ let the functions $f^\pm_X\in C^\infty(M)$ be defined by
\[f^\pm_X := |(dX)_\pm|^2.\]
As an immediate consequence of Lemma~\ref{ddX} (iv) we find

\begin{kor}
If $n_+ = n_-$, then the function $f_X := f^+_X  - f^-_X$  satisfies $\Delta f_X = f_X$, i.e. $f_X$ is a Killing potential.
\end{kor}


As an analogue of \cite[Lem.~6.12, (iii)]{NS23} we have

\begin{lem}
\label{constant}
For any Killing vector field $X$, the function given by
\[ |X|^2 + \frac{1}{2n_+n_-} z_X^2 + \frac{n}{2n_+}f^+_X + \frac{n}{2n_-}f^-_X \]
is constant on $M$.
\end{lem}
\begin{proof}
 This follows immediately by taking the differential, applying Lemma~\ref{dXprim} and Lemma~\ref{ddX} (iii), and noting that $d|X|^2=X\intprod dX$.
\end{proof}

\begin{kor}
\label{nukillingpot}
If  $n_+ = n_-$, then for any two Killing vector fields $X,Y$ we have
\[ \int_M f_X z_Y \vol = \pm2 \nu_\pm(X,X,Y). \]
\end{kor}
\begin{proof}
Recall that since $z_Y$ is a Killing potential, $\int_M z_Y \vol = 0$. Multiplying the expression in Lemma~\ref{constant} with $z_Y$ and integrating, we thus obtain
\[\int_M\left(|X|^2z_Y+\frac{1}{2n_+^2}z_X^2z_Y+f_X^+z_Y+f_X^-z_Y\right)\vol=0.\]
Also note that the terms $\int_M|X|^2z_Y\vol$ and $\int_Mz_X^2z_Y\vol$ vanish as a consequence of Lemma~\ref{cubic1} (i) and Corollary~\ref{munu} (ii), since the latter is just the polarized form $\mu(X,X,Y)$. It follows that
\[\int_Mf^+_Xz_Y\vol+\int_Mf^-_Xz_Y\vol=0\]
With Remark~\ref{polar}, we have $\int_Mf^\pm z_Y\vol=\nu_\pm(X,X,Y)$ and the assertion follows.
\end{proof}

\begin{bem}
 The above corollary amounts to the fact that the polynomials $\nu_\pm$ only differ by a sign when $n_+=n_-$. The geometric reason is the existence of an isometric antiholomorphic involution between the Grassmannians
 \[\frac{\SU(n)}{\rmS(\U(n_+)\times\U(n_-))}\longrightarrow\frac{\SU(n)}{\rmS(\U(n_-)\times\U(n_+))}\]
 that maps any $n_+$-plane in $\C^n$ to its orthogonal complement. For $n_+=n_-$ this involution interchanges the isomorphic bundles $\bbE_\pm$ and thus the polynomials $\nu_\pm$ up to a sign change due to the orientation reversal.
\end{bem}

We finish now the proof of Lemma~\ref{nu}. Let $X$ be Killing vector field such that $\nu_\pm(X,X,Y)=0$ for all Killing vectors fields $Y$. By Corollary~\ref{nukillingpot}, this is equivalent to $\Ltwoinprod{f_X}{z_Y}=0$, hence $f_X=0$, since $f_X$ and $z_Y$ are both Killing potentials. But $f_X=0$ is equivalent to $ |(dX)_+|^2 =  |(dX)_-|^2$ everywhere. We may conclude similarly to the proof of \cite[Prop.~6.16]{NS23}: of course, there are Killing vector fields with $|(dX)_+|^2 \neq |(dX)_-|^2$ at some point. For example, the fundamental vector field $\tilde X$ generated by some nonzero matrix $X\in\su(n_+)\subset\su(n)$ has $((d\tilde X)_o)_+\neq0$, but $((d\tilde X)_o)_-=0$ at the identity coset.


This shows that the cubic forms $\nu_\pm$ cannot vanish identically.

\begin{bem}
An alternative, algebraic proof of Lemma~\ref{nu} for general $n_+,n_-$ goes as follows. Consider the (non-invariant) polynomial $\nu_\pm'$ on $\isom(M,g)$ given by evaluating the integrand of $\nu_\pm$ at the identity coset, i.e.
\[\nu_\pm'(X):=|((dX)_o)_\pm|^2z_{X}(o),\qquad X\in\isom(M,g)\]
The inner product on $\bbE_\pm$ pulled back to $\su(n_\pm)\subset\k$ is some multiple of the trace form on $\su(n_\pm)$, while the Killing potential $z_X(o)$ is a constant multiple of $\langle\xi,X\rangle$, cf.~\cite[\S2.3]{HMW}, where we may choose the generator $\xi$ of $\su(n)^K$ to be
\[\xi:=\i\diag(n_-,\overset{n_+\text{ times}}{\ldots},n_-,-n_+,\overset{n_-\text{ times}}{\ldots},-n_+).\]
Thus, under the identification $\isom(M,g)\cong\g=\su(n)$, the polynomial $\nu_\pm'$ corresponds to a nonzero multiple of the cubic form $P_1\in\Sym^3\g^\ast$ given by
\begin{align*}
 P_1(X)&:=\tr(X_{\su(n_\pm)}^2)\tr(\xi X)\\
 &=\i n\left(\tr(X_{\u(n_\pm)}^2)-\frac{2n_\pm-1}{n_\pm^2}(\tr X_{\u(n_\pm)})^2\right)\tr X_{\u(n_\pm)},\qquad X\in\su(n).
\end{align*}
The key idea is now that integrating $\nu_\pm'$ over $M$ to the invariant polynomial $\nu_\pm$ amounts to projecting $P_1$ to the $G$-invariant part $(\Sym^3\g^\ast)^G$, which is by Proposition~\ref{invcubic} spanned by
\[P_0(X):=\i\tr(X^3),\qquad X\in\su(n).\]
In order to prove the nonvanishing of $\nu_\pm$ it thus suffices to show that $P_0$ and $P_1$ are not orthogonal with respect to the invariant inner product on $\Sym^3\g^\ast$ given by
\[\langle a_1a_2a_3,b_1b_2b_3\rangle=\sum_{\sigma\in\mathfrak{S}_3}\prod_{i=1}^3\langle a_i,b_{\sigma(i)}\rangle ,\qquad a_i,b_i\in\g^\ast.\]
To do that it is convenient to complex-linearly extend $P_0$ and $P_1$ to cubic forms on $\sl(n,\C)$, and to express them using the standard basis $(x_{ab})_{a,b=1}^n$ of $\gl(n,\C)^\ast$. We remark that restricted to $\sl(n,\C)$, these linear forms satisfy
\[\langle x_{ab},x_{cd}\rangle_\sl=\delta_{ac}\delta_{bd}-\frac{1}{n}\delta_{ab}\delta_{cd}.\]
The rest is a lenghty but elementary calculation.
\end{bem}

\section{The second order obstruction to integrability}
\label{sec:obstruction}


Recall from Corollary~\ref{killingparam} that the map $X\mapsto e_XJ$ from Killing vector fields to infinitesimal Einstein deformations constructed in \S\ref{sec:ied} is an isomorphism, and in particular $\varepsilon(g)=\varepsilon^+(g)$. This enables us to apply Proposition~\ref{intkaehler}. Plugging the parametrization above into the second order obstruction polynomial \eqref{psikaehler} and noting that $E=\frac{1}{2}$ by Proposition~\ref{einsteinconst}, we obtain the expression
\begin{equation}
\Psi(X)=6\Ltwoinprod{\omega\wedge de_X}{e_X\wedge de_X}-4\Ltwoinprod{e_X\wedge e_X}{e_X\wedge\omega}.
\label{psi}
\end{equation}
Using the preceding results, the goal of this section is to express the polynomial $\Psi$ on $\isom(M,g)$ in terms of the polynomials $\mu,\nu_\pm$ defined in \S\ref{sec:killingpol} and finally to show that under the idenfitication $\g\cong\isom(M,g)$, $\Psi$ corresponds to a nonzero multiple of the invariant cubic form $P_0$ that spans $(\Sym^3\g^\ast)^G$ (Proposition~\ref{invcubic}). In order to do that, we are going to analyze the two terms in \eqref{psi} separately.

\subsection{The first term}
\label{sec:firstterm}

Using Theorem~\ref{deX} (i), the integrand in the first term of \eqref{psi} is given by
\begin{align*}
\langle\omega\wedge de_X,e_X\wedge de_X\rangle=\langle\omega\wedge(X\intprod\Omega),e_X\wedge(X\intprod\Omega)\rangle
\end{align*}
Setting $h:=e_XJ$, an elementary calculation using that $e_X$ is primitive by construction and $\Omega$ is primitive by Lemma~\ref{Oprim} shows that
\[h_\ast(X\intprod\Omega)=\omega\intprod(e_X\wedge(X\intprod\Omega))\]
and thus the integrand my be rewritten as
\[\langle\omega\wedge de_X,e_X\wedge de_X\rangle=\langle h_\ast(X\intprod\Omega),X\intprod\Omega\rangle.\]
Substituting in the definitions of $e_X$ and $\Omega$ and expanding the sum, we are faced with terms of the following type.

\begin{lem}
\label{ugly}
For any $v\in TM$ and $h\in\Sym^2TM$ such that $hJ\in\bbE_\pm$, we have
\begin{enumerate}[\upshape(i)]
 \item $\displaystyle\langle h_\ast(v\intprod\Omega_\pm),v\intprod\Omega_\pm\rangle=4\frac{n_\pm^3n_\mp-n_\pm^2-5n_\pm n_\mp-3}{n_+^2n_-^2}h(v,v)$,
 \item $\displaystyle\langle h_\ast(v\intprod\Omega_\mp),v\intprod\Omega_\mp\rangle=4\frac{n_\pm^3n_\mp+3n_\pm^2-n_\pm n_\mp-3}{n_+^2n_-^2}h(v,v)$,
 \item $\displaystyle\langle h_\ast(v\intprod\Omega_\pm),v\intprod\Omega_\mp\rangle=\langle h_\ast(v\intprod\Omega_\mp),v\intprod\Omega_\pm\rangle=-4\frac{n_+^2n_-^2-n_\pm^2-3n_\mp^2+3}{n_+^2n_-^2}h(v,v)$.
\end{enumerate}
\end{lem}
\begin{proof}
 Suppose that $hJ\in\bbE_+\oplus\bbE_-$. First, we note that since $h_\ast$ is a derivation,
 \begin{align*}
 h_\ast(v\intprod\Omega_\pm)&=v\intprod h_\ast\Omega_\pm-h(v)\intprod\Omega_\pm,\\
 &=2\sum_a\left((v\intprod\{h,\omega_a^\pm\})\wedge\omega_a^\pm+(v\intprod\omega_a^\pm)\wedge\{h,\omega_a^\pm\}-(h(v)\intprod\omega_a^\pm)\wedge\omega_a^\pm\right).
 \end{align*}
 It follows that
 \begin{equation}
  \langle h_\ast(v\intprod\Omega_\pm),v\intprod\Omega_\pm\rangle=-2\sum_a\left\langle(\omega_a^\pm\intprod\Omega_\pm)h\omega_a^\pm v+\omega_a^\pm(\{h,\omega_a^\pm\}\intprod\Omega_\pm)v,v\right\rangle.
  \label{firststep}
 \end{equation}
 By virtue of  Lemma~\ref{Cendo} and Lemma~\ref{inOp}, we have
 \[\omega_a^\pm\intprod\Omega_\pm=\frac{2(n_+n_-+1)}{n_+n_-}\omega_a^\pm,\qquad\{h,\omega_a^\pm\}\intprod\Omega_\pm=2\{h,\omega_a^\pm\}_\pm+2C_\pm(\{h,\omega_a^\pm\})\]
 and we note that Corollary~\ref{anticomm1} and Corollary~\ref{anticomm2} imply
 \[\{h,\omega_a^\pm\}_\pm=\begin{cases}
                           \{h,\omega_a^\pm\}-\frac{2}{n_+n_-}\langle hJ,\omega_a^\pm\rangle\omega,&h\in\bbE_\pm J,\\
                           0,&h\in\bbE_\mp J.
                          \end{cases}
 \]
 We substitute this into \eqref{firststep} and continue analyzing the occurring terms. First,
 \begin{align*}
  \sum_a(\omega_a^\pm\intprod\Omega_\pm)h\omega_a^\pm&=\frac{2(n_+n_-+1)}{n_+n_-}C_\pm(h).
 \end{align*}
 Second, using the equivariance of $C_\pm$ from Lemma~\ref{Cendo} together with Lemma~\ref{casvalues},
 \begin{align*}
  \sum_a\omega_a^\pm C_\pm(\{h,\omega_a^\pm\})&=\sum_{a,b}\omega_a^\pm\omega_b^\pm(h\omega_a^\pm+\omega_a^\pm h)\omega_b^\pm\\
  &=\sum_{a,b}(\omega_a^\pm\omega_b^\pm[h,\omega_a^\pm]\omega_b^\pm+2\omega_a^\pm\omega_b^\pm\omega_a^\pm h\omega_b^\pm)\\
  &=\sum_a(\omega_a^\pm C_\pm([h,\omega_a^\pm])+2C_\pm(\omega_a^\pm)h\omega_a^\pm)\\
  &=\sum_a(\omega_a^\pm[C_\pm(h),\omega_a^\pm]+\tfrac{2}{n_+n_-}\omega_a^\pm h \omega_a^\pm)\\
  &=\sum_a(\omega_a^\pm C_\pm(h)\omega_a^\pm-(\omega_a^\pm)^2C_\pm(h))+\tfrac{2}{n_+n_-}C_\pm(h)\\
  &=C_\pm^2(h)+\Cas^{\su(n_\pm),\langle\cdot,\cdot\rangle_{\Lambda^2}}_\m C_\pm(h)+\tfrac{2}{n_+n_-}C_\pm(h)\\
  &=C_\pm^2(h)+\tfrac{n_\pm^2+1}{n_+n_-}C_\pm(h),
 \end{align*}
 while (in the case of $h\in\bbE_\pm J$)
 \begin{align*}
  \sum_a\omega_a^\pm\{h,\omega_a^\pm\}&=\sum_a(\omega_a^\pm h \omega_a^\pm+(\omega_a^\pm)^2h)\\
  &=C_\pm(h)-\Cas^{\su(n_\pm),\langle\cdot,\cdot\rangle_{\Lambda^2}}_\m h=C_\pm(h)-\frac{n_\pm^2-1}{n_+n_-}h,\\
  \sum_a\langle hJ,\omega_a^\pm\rangle\omega_a^\pm J&=-h.
 \end{align*}
 Putting this together and applying Lemma~\ref{Cendo}, we obtain
 \begin{align*}
  \langle h_\ast(v\intprod\Omega_\pm),v\intprod\Omega_\pm\rangle&=\begin{cases}
                                                                   4\dfrac{n_\pm^3n_\mp-n_\pm^2-5n_\pm n_\mp-3}{n_+^2n_-^2}h(v,v),&h\in\bbE_\pm J,\\
                                                                   4\dfrac{n_\pm^3n_\mp+3n_\pm^2-n_\pm n_\mp-3}{n_+^2n_-^2}h(v,v),&h\in\bbE_\mp J.\\
                                                                  \end{cases}
 \end{align*}
 This proves (i) and (ii). In order to show (iii), we first note that it suffices to analyze $\langle h_\ast(v\intprod\Omega_\pm),v\intprod\Omega_\mp\rangle$ with $h\in\bbE_\pm J$ since $h_\ast$ is self-adjoint. We rerun the calculation above with an $\Omega_\mp$ on the right hand side of the inner product. Again combining Lemma~\ref{Cendo} with Lemma~\ref{inOp} and Corollary~\ref{anticomm2}, we find
 \[\omega_a^\pm\intprod\Omega_\mp=-\frac{2(n_\mp^2-1)}{n_+n_-}\omega_a^\pm,\qquad\{h,\omega_a^\pm\}\intprod\Omega_\mp=2C_\mp(\{h,\omega_a^\pm\}).\]
 As above, we unfold
 \begin{align*}
  \sum_a\omega_a^\pm C_\mp(\{h,\omega_a^\pm\})&=\sum_{a,b}\omega_a^\pm\omega_b^\mp(h\omega_a^\pm+\omega_a^\pm h)\omega_b^\mp\\
  &=\sum_a\omega_a^\pm C_\mp([h,\omega_a^\pm])+2\sum_bC_\pm(\omega_b^\mp)h\omega_b^\mp\\
  &=C_\pm(C_\mp(h))+\Cas^{\su(n_\pm),\langle\cdot,\cdot\rangle_{\Lambda^2}}_\m C_\mp(h)-2\tfrac{n_\pm^2-1}{n_+n_-}C_\mp(h)\\
  &=C_\pm(C_\mp(h))-\tfrac{n_\pm^2-1}{n_+n_-}C_\mp(h)
 \end{align*}
 and combining everything we obtain
 \[\langle h_\ast(v\intprod\Omega_\pm),v\intprod\Omega_\mp\rangle=-4\frac{n_+^2n_-^2-n_\pm^2-3n_\mp^2+3}{n_+^2n_-^2}h(v,v).\]
\end{proof}


Assembling the preceding identities, we now obtain an expression for the first term of the obstruction polynomial \eqref{psi}.

\begin{kor}
\label{firstterm}
The integrand in the first term of \eqref{psi} is
\[\langle\omega\wedge de_X,e_X\wedge de_X\rangle=k_+(dX)_+(JX,X)-k_-(dX)_-(JX,X),\]
where the constants $k_\pm$ are given by
\[k_\pm=\frac{(n_\mp^2-1)^2(n_\pm^3n_\mp+n_\pm^2-5n_+n_-+3)}{n_\pm^2n_\mp^3}.\]
\end{kor}
\begin{proof}
Continuing the discussion at the beginning of \S\ref{sec:firstterm}, we may combine the definition of $\Omega$ with Lemma~\ref{ugly} and find
\[\langle\omega\wedge de_X,(dX)_\pm\wedge de_X\rangle=\frac{(n_\mp^2-1)(n_\pm^3n_\mp+n_\pm^2-5n_+n_-+3)}{n_+^2n_-^2}(dX)_\pm(JX,X).\]
Combining this with the definition of $e_X$ yields the assertion.
\end{proof}

Using the identities in Corollary~\ref{munu} and Lemma~\ref{cubic2}, it is possible to express the integral of the above term in terms of the cubic forms $\mu$ and $\nu_+$. We record the following:

\begin{kor}
\label{firstterm2}
The first term of \eqref{psi} is given by
\begin{align*}
 \Ltwoinprod{\omega\wedge de_X}{e_X\wedge de_X}=\begin{dcases}
    \frac{(n_+^2-1)(n_-^2-1)}{n_+^3n_-^3(n_+-n_-)}\ell_1\mu(X),&n_+\neq n_-,\\
    -2\frac{(n_+^2-1)^3(n_+^2-3)}{n_+^5}\nu_+(X),&n_+=n_-,
   \end{dcases}
\end{align*}
with the constant $\ell_1$ given by
\[\ell_1:=n_+^3n_-^3-3n_+^3n_--3n_+n_-^3+n_+^2n_-^2+n_+^2+n_-^2+5n_+n_--3.\]
\end{kor}

\subsection{The second term}

The second term in \eqref{psi} does not depend on the derivatives of $e_X$. Thanks to the algebraic preparation in \S\ref{sec:commanticomm}, it is therefore much easier to handle than the first one.

\begin{lem}
\label{tripleprod}
For $\alpha,\beta,\gamma\in\bbE_+\oplus\bbE_-$, we have
\[\langle\alpha\wedge\beta,\gamma\wedge\omega\rangle=\langle\{\alpha,\beta\}J,\gamma\rangle,\]
and this vanishes unless either $\alpha,\beta,\gamma\in\bbE_+$ or $\alpha,\beta,\gamma\in\bbE_-$.
\end{lem}
\begin{proof}
We begin with noting that
\[\langle\alpha\wedge\beta,\gamma\wedge\omega\rangle=\langle\omega\intprod(\alpha\wedge\beta),\gamma\rangle\]
and since $\alpha,\beta\perp\omega$ and $\alpha,\beta$ commute with $J$, Lemma~\ref{wedgein} implies that
\[\omega\intprod(\alpha\wedge\beta)=\alpha J\beta+\beta J\alpha=\{\alpha,\beta\}J.\]
The vanishing statements now follow from Corollary~\ref{anticomm1} resp.~\ref{anticomm2}.
\end{proof}

Substituting the definition of $e_X$ into $\langle e_X\wedge e_X,e_X\wedge\omega\rangle$, expanding the sum and applying the above lemma immediately yields the following.

\begin{kor}
\label{secondterm}
The integrand in the second term of the obstruction polynomial is
\[\langle e_X\wedge e_X,e_X\wedge\omega\rangle=\frac{2(n_-^2-1)^3}{n_-^3}\langle(dX)_+^2J,(dX)_+\rangle-\frac{2(n_+^2-1)^3}{n_+^3}\langle(dX)_-^2J,(dX)_-\rangle.\]
\end{kor}

As before, we reduce the integral of this term to multiples of $\mu$ and $\nu_+$ by means of Corollary~\ref{munu} and Lemma~\ref{cubic2}.

\begin{kor}
\label{secondterm2}
The second term of \eqref{psi} is given by
\begin{align*}
 \Ltwoinprod{e(X)\wedge e(X)}{e(X)\wedge\omega}=\begin{dcases}
    \frac{(n_+^2-1)(n_-^2-1)}{n_+^3n_-^3(n_+-n_-)}\ell_2\mu(X),&n_+\neq n_-,\\
    -2\frac{(n_+^2-1)^3(n_+^2-4)}{n_+^5}\nu_+(X),&n_+=n_-,
   \end{dcases}
\end{align*}
with the constant $\ell_2$ given by
\[\ell_2:=n_+^3n_-^3-4n_+^3n_--4n_+n_-^3+2n_+^2n_-^2+n_+^2+n_-^2+7n_+n_--4.\]
\end{kor}

\subsection{The full obstruction polynomial}

Using the expressions in Corollary~\ref{firstterm2} and Corollary~\ref{secondterm2} for the individual terms of \eqref{psi}, it is now a straightforward matter of calculation to reduce $\Psi$ itself to a multiple of $\mu$ (if $n_+\neq n_-$) or $\nu_+$ (if $n_+=n_-=\frac{n}{2}$), respectively. We obtain the following result.

\begin{satz}
\label{psiexpression}
The obstruction polynomial $\Psi$ may be expressed as
\[\Psi=\begin{dcases}
\frac{2(n_+^2-1)^2(n_-^2-1)^2(n_+n_--1)}{n_+^3n_-^3(n_+-n_-)}\mu&\text{if }n_+\neq n_-,\\
-\frac{(n^2-4)^4}{2n^5}\nu_+&\text{if }n_+=n_-.
\end{dcases}\]
\end{satz}

Notice that for $n_-=2$ this agrees with the result\footnote{This refers to the updated version of \cite{NS23}, containing the correct formula \eqref{psikaehler}, which is in preparation at the time of writing.} in \cite[Thm.~6.24]{NS23} (up to a scale factor of $\left(\frac{2}{3}\right)^3$ stemming from a different parametrization, and the notation $m=2n_+$).

We recall from \S\ref{sec:killingpol} that the cubic form $\nu_+$ is nonzero, and so is $\mu$ provided $n_+\neq n_-$. Thus we conclude as follows.

\begin{kor}
\label{psinonzero}
The obstruction polynomial $\Psi$ is nonzero for $n_+,n_-\geq2$.
\end{kor}

Theorem~\ref{main} is now proved as follows. Since $\Psi$ is an element of $(\Sym^3\g^\ast)^G$, which is spanned by $P_0$ (Proposition~\ref{invcubic}), it must thus be a nonzero multiple of $P_0$. As a consequence, the set of infinitesimal Einstein deformations integrable to second order
corresponds under the parametrization $X\mapsto e_XJ$ to the variety $\quadric\subset\g$ described in Proposition~\ref{variety}. If $n$ is odd, this variety is just the origin -- thus no infinitesimal Einstein deformation is integrable, and $g$ is rigid.

\section*{Acknowledgments}

The first named author acknowledges support by the Procope project no.~\textbf{48959TL} and by the BRIDGES project funded by ANR grant no.~\textbf{ANR-21-CE40-0017}. The second named author was supported by the Special Priority Program \textbf{SPP 2026 ``Geo\-me\-try at Infinity''} funded by the DFG.

We would like to express our thanks to Stuart J.~Hall and Klaus Kröncke for helpful exchanges, and to Paul-Andi Nagy for his suggestions and for inspiring this work.

\end{document}